\documentclass[reqno]{amsart}
\newcommand \datum {October 8, 2023}

\usepackage{amssymb,latexsym}
\usepackage{amsmath}
\usepackage{hyperref}
\usepackage{url}
\usepackage{graphicx}
\usepackage[dvipsnames]{xcolor}
\usepackage{enumerate}
\usepackage{tikz, color}
\numberwithin{equation}{section}
\theoremstyle{plain}
 \newtheorem{theorem}{Theorem}[section]
  
 \newtheorem{lemma}[theorem]{Lemma}
 \newtheorem{proposition}[theorem]{Proposition}
 
 \newtheorem{observation}[theorem]{Observation}

\theoremstyle{definition}
 \newtheorem{definition}[theorem]{Definition}

 \newtheorem{remark}[theorem]{Remark}
\theoremstyle{remark}

\newcommand \ideal [1]{\mathord\downarrow #1}
\newcommand \pideal [2]{\mathord\downarrow_{#1}#2} 
\newcommand \pfilter [2]{\mathord\uparrow_{#1}#2}
\newcommand \filter[1]{\mathord\uparrow #1}

\newcommand \ibinom[2] {\textup C_{\textup{bin}}(#1,#2)}

\newcommand \Symn {\textup{Sym}_n}
\newcommand \reqref[1]  {\mathrel{\overset{\eqref{#1}}=}}
\newcommand \rleqref[1] {\mathrel{\overset{\eqref{#1}}\leq}}
\newcommand \embeddable { \mathrel{  \raisebox{-2.5pt}{$\overset{\scriptscriptstyle{\textup{exists}}}\hookrightarrow$}   }}
\newcommand\embeds {\hookrightarrow}
\newcommand\vv{\vec v}
\newcommand\vpv{\vec v^\bullet}
\newcommand\vX{\vec X}
\newcommand\vpX{\vec X^\bullet}
\newcommand\sima{a}
\newcommand\bulla{a^\bullet}
\newcommand\bullX{X^\bullet}
\newcommand\bulli{i^\bullet}
\newcommand\bullv{v^\bullet}
\newcommand\setc{\overline c}
\newcommand\setd{\overline d}
\newcommand\sete{\overline e}
\newcommand\setz{\overline z}
\newcommand\seta{\overline a}
\newcommand\bullseta{\overline a^\bullet}
\newcommand \qum{h}

\newcommand\dst{3.4} 
\newcommand \bcirc[1] {\fill[black] (#1) circle (5pt);
  \draw [black,thick] (#1)  circle [radius=5pt]}
\newcommand \lcirc[1] {\fill[red] (#1) circle (2pt);
  \draw [red,thick] (#1)  circle [radius=2pt]}
\newcommand \wcirc[1] {\fill[white] (#1) circle (5pt);
  \draw [black,thick] (#1)  circle [radius=5pt]}
\newcommand \nodetxt[2] {\draw [white,fill] (#1) circle [radius=0.0pt] node [black,below=7pt] {#2} } 
\newcommand \nodeutxt[2] {\draw [white,fill] (#1) circle [radius=0.0pt] node [black,above=-1pt] {#2} } 
\newcommand \nodeuutxt[2] {\draw [white,fill] (#1) circle [radius=0.0pt] node [black,above=1pt] {#2} } 
\newcommand \nodettxt[2] {\draw [white,fill] (#1) circle [radius=0.0pt] node [red,above=-7pt] {\tbf{#2}}} 

\newcommand\hot{20pt}
\newcommand\vot{20pt}
\newcommand\het{20pt}
\newcommand\wit{20pt}
\newcommand\gap{8pt}

\newcommand \fbx[2]{A^{(#1)}_{#2}}
\newcommand \Dn[1] {\textup{Dn}(#1)}
\newcommand \meetgen[1] {[#1]_\wedge}
\newcommand \chain [1]{\mathsf C_{#1}}
\newcommand \Pow [1]{\mathsf P(#1)}
\newcommand \nPow {\Pow{[n]}}
\newcommand \loS  {{_{\textup{lo}}S}}
\newcommand \upS  {{^{\textup{up}\kern -1pt}S}}
\newcommand \uppS  {{^{\textup{up}+\kern -1pt}S}}
\newcommand \lupS[1]  {\lint{{^{\textup{up}\kern -1pt}S(#1)}}}
\newcommand\Sp {S}
\newcommand\fsb {f_{\textup{sb}}}
\newcommand\Asp {S^{\ast}}
\newcommand\afsb {f^{\ast}_{\textup{sb}}}
\newcommand\gmin [1]{G_{\textup{min}}(#1)}
\newcommand \Jir [1]{J(#1)}

\newcommand \vp{\vec\pi}
\newcommand \lpos {\textup{Lp}}
\newcommand \iset {\textup{Is}}
\newcommand\lint[1]{\lfloor #1\rfloor}
\newcommand\uint[1]{\lceil #1\rceil}
\newcommand\FS [1] {F_{\textup{meet}}(#1)}
\renewcommand \phi{\varphi}
\newcommand \Nnul {{\mathbb N_0}}
\newcommand \Nplu {{\mathbb N^+}}
\newcommand{\tbf}{\textbf}
\newcommand{\set}[1]{\{#1\}}

\newcommand \jour[1]{}

\newcommand \nothing[1] {}
\newcommand \onlymiktex[1]{#1} 
\newcommand \red[1]{{\textcolor{red}{#1}\color{black}}}

\begin{document}

\title[Sperner theorems and minimum generating sets]
{Sperner theorems for unrelated copies of some partially ordered sets in a powerset lattice 
and minimum generating sets 
of  powers of distributive lattices}

\author[G.\ Cz\'edli]{G\'abor Cz\'edli}
\email{czedli@math.u-szeged.hu}
\urladdr{http://www.math.u-szeged.hu/~czedli/}
\address{University of Szeged, Bolyai Institute. 
Szeged, Aradi v\'ertan\'uk tere 1, HUNGARY 6720}

\begin{abstract} For a finite poset (partially ordered set) $U$ and a natural number $n$, let $\Sp(U,n)$ denote the largest number of pairwise unrelated copies of $U$ in the powerset lattice (AKA subset lattice) of an $n$-element set. If $U$ is the singleton poset, then $\Sp(U,n)$ was determined by E.\ Sperner in 1928; this result is well known in extremal combinatorics. Later, exactly or asymptotically,  Sperner's theorem was extended to  other posets   by A.\,P.\ Dove, J.\,R.\ Griggs, G.\,O.\,H.\ Katona, D.\ Nagy,  J.\ Stahl, and W.\,T.\,Jr.\ Trotter.
We determine $\Sp(U,n)$ for all finite posets with 0 and 1, and we give reasonable estimates for the ``V-shaped'' 3-element poset  and the 4-element poset with 0 and three maximal elements.
 
For a lattice $L$, let $\gmin L$ denote the minimum size of generating sets of $L$. We prove that if $U$ is the poset of  the join-irreducible elements of a finite distributive lattice $D$, then the function $k\mapsto \gmin{D^k}$ 
is the left adjoint of the function $n\mapsto \Sp(U,n)$. This allows us to determine $\gmin{D^k}$ in many cases. E.g., for a 5-element distributive lattice $D$,  $\gmin{D^{2023}}=18$ if $D$ is a chain and 
 $\gmin{D^{2023}}=15$ otherwise.
 
It follows that large direct powers of small distributive lattices are appropriate for our 2021 cryptographic authentication protocol. 
\end{abstract}

\thanks{This research was supported by the National Research, Development and Innovation Fund of Hungary, under funding scheme K 138892.  \hfill{\red{\tbf{\datum}}}}




\maketitle

\section{Introduction}\label{sect:intro}
\subsection{Targeted readership}
This paper belongs both to extremal combinatorics and lattice theory, and  it is intended to be self-contained for those who know the concept of a free semilattice, that of a distributive lattice, and the relation between lattice orders and lattice operations. That is, apart from some basic combinatorial facts that are always taught for B.Sc.\ students, the reader is assumed to be familiar only with 
some  facts and concepts that are often taught in M.Sc. courses.

\subsection{Purpose and outline} Our main goal is to establish a \emph{bridge} between the combinatorial topic  of  Sperner (type) theorems 
and the lattice theoretical topic of  minimum generating sets of finite lattices; this goal is accomplished by 
Theorem \ref{thm:main} in Section \ref{sect:bridge}. 
If we start from  the Sperner (type) theorems proved  by Griggs, Stahl, and Trotter \cite{griggsatall},
Dove and Griggs \cite{dovegriggs}, and Katona and Nagy \cite{KatonaNagy},  then the just-mentioned ``bridge'' can lead only to asymptotic results, in which we are less interested, or to  
rather special distributive lattices. Hence,  we generalize their Sperner theorems  in a modest way, see Observation \ref{obs:nsTgtfblS}, and we give reasonable estimates for a particular case; see Proposition \ref{prop:Wposet}.

A poset (that is, partially ordered set) $U$ is said to be \emph{bounded} if it has a smallest element $0=0_U$ and a largest element $1=1_U$; these elements are uniquely determined if they both exist. 
In Section \ref{sect:exact}, we give an \emph{exact formula} for the maximum number of pairwise unrelated isomorphic copies of a finite \emph{bounded} poset  among the subsets of an $n$-element set; see Observation \ref{obs:nsTgtfblS}, which is an easy  generalization of a result of Griggs, Stahl, and Trotter \cite{griggsatall} from chains to bounded posets.
The situation becomes more exiting in Section \ref{sect:loweruppergeneral}, where we present estimates for two particular posets, $V$ and $W$ given in Figure \ref{figone}.

The search for small generating sets has more than half a century long history. Indeed, this topic goes back  (at least) to Gelfand an Ponomarev \cite{gelfand}; see Z\'adori \cite{zadori} for details of their result on subspace lattices. For small generating sets in some other lattices,
see also the introductions and the bibliographic sections of  Cz\'edli \cite{czgDEBRauth} and \cite{czgboolegen}.
Recently in \cite{czgDEBRauth} and \cite{czgboolegen}, we have pointed out that large lattices and large powers of (small) lattices can have applications in cryptography and authentication provided that they have small generating sets. This had led to the original motivation of the present paper: we wanted to determine how much 
elements are needed to generate a large direct power of a small distributive lattice. 

Even though we prove only estimates rather than exact Sperner theorems in Section \ref{sect:loweruppergeneral}, they are sufficient 
to determine the minimum number of generators of direct powers of the corresponding distributive lattices with quite good accuracy and, in most of the cases, exactly; this will be formulated in \eqref{eq:mLptpsGPs} and 
 exemplified explicitly by   \eqref{eq:mFljZrstlWlFj} and implicitly by  all collections of concrete data displayed in the paper. 
Note that even less  accuracy would be satisfactory from cryptographic point of view, in which 
the role of a small \emph{minimum} number of generators is to indicate that there are many \emph{small generating sets}. 
Hence, in addition to the exact lattice theoretical results that we can  obtain by  combining Theorem \ref{thm:main} with  Observation  \ref{obs:nsTgtfblS} or \eqref{eq:knknRswt}, Section \ref{sect:loweruppergeneral} also offers new possibilities for the cryptographic authentication protocol given in \cite{czgDEBRauth} and \cite{czgboolegen}.

\section{A bridge between combinatorics and lattice theory}\label{sect:bridge}
Except for the sets $\Nplu:=\set{1,2,3,\dots}$ and $\Nnul:=\set{0}\cup \Nplu\kern-1pt$, all sets and structures in this paper are assumed to be \emph{finite} even when this is not explicitly mentioned. 

Next, we recall some concepts and notations, and introduce a  few new ones.
For a real number $x$, the \emph{lower integer part} and the \emph{upper integer part} of $x$ are denoted by $\lint x$ and $\uint x$, respectively. For $n\in \Nnul$, note the rule: $\lint{n/2}+\uint{n/2}=n$. 
A  function $f\colon \Nnul\to\Nnul$ is \emph{non-bounded} if for each $k\in\Nnul$, there exists an $n\in\Nnul$ such that $f(n)>k$. 
For a non-bounded function $f\colon \Nnul\to\Nnul$, the \emph{left adjoint} $f^\ast$ of $f$ is the function
\begin{equation}
f^\ast\colon \Nnul\to\Nnul\text{ defined by }k\mapsto \min\set{n\in\Nnul: k\leq f(n)}. 
\label{eq:gPnmfnktPvL}
\end{equation}
(The terminology ``left adjoint'' is explained by categorified posets, but we do not need this fact.) If $f(x)\leq f(y)$ holds whenever $x\leq y$, then $f$ is an \emph{increasing function}. For $\Nnul\to\Nnul$ functions $f_1$ and $f_2$,  $f_1\leq f_2$ means that $f_1(x)\leq f_2(x)$ holds for every $x\in \Nnul$. The following lemma follows in a straightforward way from definitions and it belongs to folklore, so we do not prove it in the paper.

\begin{lemma}\label{lemma:wBprznJTrh}
If $f$, $f_1$, and $f_2$ are increasing non-bounded $\Nnul\to \Nnul$ functions then so are their left adjoints. Furthermore, for all  $n,k\in\Nnul$,   
\begin{align}
&k\leq f(n)\,\,\text{ if and only if }\,\, f^\ast(k)\leq n,
\label{eq:smZstmrCvr} \\
&k > f(n)\,\,\text{ if and only if }\,\, f^\ast(k)> n,
\label{eq:bDVLtlkHvj} \\
&f(n)=\max\set{y\in\Nnul: f^\ast(y)\leq n}
\label{eq:hmNmfMnStHhKnrP},\text{ and}\\
&f_1\leq f_2\,\,\text{ if and only if }\,\, f_2^\ast \leq f_1^\ast.
\label{eq:srLtMksZkrmglDn}
\end{align}
\end{lemma}

For a poset  $U$ and a natural number $k\in\Nplu$, let $kU=(kU,\leq)$ denote the \emph{cardinal sum of $k$ isomorphic copies} of $U$. That is, if $(U_1,\rho_1)$, \dots, $(U_k;\rho_k)$ are pairwise disjoint isomorphic copies of $U=(U;\leq)$, then $(kU;\leq):=(U_1\cup\dots \cup U_k; \rho_1\cup\dots \cup \rho_k)$. 
Then for $x\in U_i$ and $y\in U_j$, if $i\neq j$, then neither $x\leq y$ nor $y\leq x$, that is, $x$ and $y$ are \emph{incomparable}, in notation,  $x\parallel y$.  
In other words, $U_i$ and $U_j$ are \emph{unrelated} for  $i\neq j$. 
We obtain the (Hasse) diagram of $kU$ by  putting $k$ copies of the diagram of $U$ side by side.
For $k\in \Nnul$, the $(k+1)$-element chain will be denoted by $\chain k$. Note that $k\chain 0$ is the \emph{$k$-element antichain}. For $n\in\Nplu$, $[n]$ will stand for the set $\set{1,\dots, n}$ while $[0]:=\emptyset$. 
For a set $A$, the \emph{powerset lattice} (also called the \emph{subset lattice}) of $A$ is the lattice $(\set{X: X\subseteq A};\subseteq)$.  In this lattice, which we denote by  $\Pow A$ or $(\Pow A;\subseteq)$, the operations
$\vee$ and $\wedge$ are $\cup$ and $\cap$, respectively. 
For an element $y$ in a poset $U$, 
we denote $\set{x\in U: x\leq y}$ by $\ideal y$ or, if confusion threatens, by  $\pideal U y$. Similarly, $\filter y$ and $\pfilter U y$ stand for $\set{x\in U: y\leq x}$. 
For posets $U_1$ and $U_2$ and a function $\phi\colon U_1\to U_2$, $\phi$ is an \emph{order embedding} if for all $x,y\in U_1$, $x\leq y\iff \phi(x)\leq \phi(y)$. 
Let $\phi\colon U_1\embeds U_2$ denote that $\phi$ is an order embedding. Furthermore,  let $U_1\embeddable U_2$  denote that there exists an order embedding $\phi\colon U_1\embeds U_2$.
For example, if $U$ is a poset, then the function $U\to \Pow U$ defined by $y\mapsto\pideal U y$ is an order embedding. Thus,
\begin{equation}
\text{for any poset }U, \text{ we have that }U\embeddable\Pow{|U|}.
\label{eq:vnkRkmdHs}
\end{equation}
If $U_1\subseteq U_2$ and the function $U_1\to U_2$ defined by $x\mapsto x$ is an order-embedding, then $U_1$ is a \emph{subposet} of $U_2$; this fact is denoted by $U_1\leq U_2$.
A poset cannot be empty by definition; the only exception is that for every poset $U$, $0U$ is a subposet of (and is embedded into) any other poset; the following definition needs this convention.

\begin{definition} \label{def:SpSnl}
Let $U$ be a finite poset. For $k,n\in\Nnul$, let 
\begin{align}
\Sp(U,n)&:=\max\set{k\in\Nnul: k U \embeddable \nPow }\text{ and}
\label{eq:zsdmznrmVck}\\
\Asp(U,k)&:=\min\set{n\in \Nnul: k U\embeddable\nPow } =  \min\set{n\in\Nnul: k\leq \Sp(U,n)};
\label{eq:rtrkfjnhTkzkc}
\end{align}
\eqref{eq:vnkRkmdHs} implies that the definition ``$:=$'' in \eqref{eq:rtrkfjnhTkzkc} makes sense.
For $n\in\Nplu$, let
\begin{equation}
\fsb(n):={n\choose{\lint{n/2}}}\,\,
\text{ and }\,\,\afsb(k):=\min\set{n\in\Nplu: k\leq \fsb(n)}. 
\label{eq:snuln}
\end{equation}
For the sake of  better outlook and optical readability, let us agree that in in-line formulas, we often write $\ibinom {m}{t}$  instead of $\binom m t$; especially when $m$ or $t$ are complicated expressions with subscripts. With this convention, $\fsb(n)=\ibinom n{\lint{n/2}}$.
\end{definition}

\begin{remark}\label{rem:RmtVznkRm}
The notation in Definition \ref{def:SpSnl} is coherent with \eqref{eq:gPnmfnktPvL} since the functions $\Asp(U,-)\colon \Nnul\to \Nnul$ defined by $n\mapsto \Asp(U,n)$ and $\afsb$ are the left adjoints of the functions $\Sp(U,-)\colon \Nnul\to \Nnul$ defined by $n\mapsto \Sp(U,n)$ and $\fsb$, respectively. (For $\Asp(U,-)$, this follows immediately from $kU  \embeddable \nPow  \iff k\leq \Sp(U,n)$.) 
\end{remark}

The remark above enables us to  benefit from Lemma \ref{lemma:wBprznJTrh}. Note that the notation  $\fsb$ comes from  \underline Sperner's original \underline Binomial coefficient as a  \underline Function.
For subsets $X$ and $Y$ of $[n]$, using the terminology of Griggs, Stahl, and Trotter \cite{griggsatall}, we say that $X$ and $Y$ are \emph{unrelated} if $x\parallel y$ for all $x\in X$ and $y\in Y$. So $\Sp(U,n)$ is the maximum number of pairwise unrelated isomorphic copies of $U$ in $\nPow$.

With the notation introduced in Definition \ref{def:SpSnl}, Sperner's Theorem from \cite{sperner} asserts that $\Sp(\chain 0,n)=\fsb(n)$ while 
a Sperner theorem (i.e., a Sperner-type theorem) proved by Griggs, Stahl and Trotter \cite[Theorem 2]{griggsatall} asserts that 
\begin{equation}
\text{for }t\in\Nplu,\quad
\Sp(\chain t,n)=\fsb(n-t),\text{ that is, }\Sp(\chain t,n)= {n-t\choose{\lint{(n-t)/2}}}.
\label{eq:griggs}
\end{equation}
Note that, by convention, $\fsb(n-t)=0$ for $n<t$. For later reference, some values of $\Sp(\chain4,n)$ are as follows:
\begin{equation}
\begin{tabular}{l|r|r|r|r|r|r}
$n$ &17&18& 2024& 2025 &2026  \cr
\hline
$\Sp(\chain4,n)$ & 1\,716 & 3\,432 &
$2.137  \cdot 10^{606}$
&$4.272\cdot 10^{606}$ &
$8.544 \cdot 10^{606}$
 \cr
\end{tabular}\,.
\label{eq:tvgrHptnKhnjN}
\end{equation}

The \emph{length} of a finite poset $U$ is the largest $t$ such that $\chain t$ is a subposet of $U$. 
The result cited in \eqref{eq:griggs} has been generalized by Katona and Nagy \cite[Theorem 4.3]{KatonaNagy} to the following one.
\begin{equation}
\parbox{8cm}
{If $U$ is a finite poset of length $t$ such that $\Asp(U,1)=t$ then, for every $n\in\Nnul$,  $\Sp(U,n)=\fsb(n-t)$.}
\label{eq:knknRswt}
\end{equation}

A \emph{proper sublattice} of a lattice $L$ is a nonempty subset $X$ of $L$ such that $ X\neq L$ and $X$ is closed with respect to $\vee$ and $\wedge$. A subset $Y$ of $L$ is a \emph{generating set} of $L$ if no proper sublattice of $L$ includes $Y$.
As $L$ is assumed to be finite, the \emph{least size of a generating set} of $L$ makes sense; we denote it by 
\begin{equation}
\gmin L:=\min\set{|Y|: Y \text{ is a generating set of }L}.
\label{eq:gmindf}
\end{equation} 
In the $k$-th direct power $L^k:=L\times\dots\times L$ ($k$-fold direct product) of $L$, the lattice operations are performed  component-wise; we are interested in $\gmin {L^k}$ for some distributive lattices $L$. 
The set of \emph{join-irreducible} elements of $L$ is denoted by $\Jir L$; by definition, $x\in L$ belongs to $\Jir L$ if and only if $x$ covers exactly one element; in particular, the smallest element  $0=0_L$ of $L$ is not in $\Jir L$. With the order inherited from $L$,  $\Jir L=(\Jir L;\leq)$ is a poset.  
Now  that we have
\eqref{eq:gmindf} and Definition \ref{def:SpSnl}, we can formulate the main result of the paper.

\begin{theorem}\label{thm:main} 
If $D$ is a finite distributive lattice and $2\leq k\in\Nplu$, then 
$\gmin{D^k}=\Asp(\Jir D,k)$.
\end{theorem}

\begin{proof} For $t\in \Nplu$, denote by $\FS t=\FS{x_1,\dots,x_t}$ the free meet-semilattice with free generators $x_1$, \dots, $x_t$. We know from folklore and from \S4 in 
Page 240  of McKenzie, McNulty  and Taylor \cite{mcKmcNT} (and it is not hard to see) that $\FS t$ is a subposet of $\Pow {[t]}$; in fact, $\FS t$ is (order isomorphic to) $\Pow{[t]}\setminus\set{[t]}$.

Let $U:=\Jir D$. 
With $U_1:=U\times\set 0\times\dots \times \set0$, \dots, $U_k:=\set 0\times\dots \times \set0\times U$, it is clear that $U_1\cup\dots\cup U_k\subseteq \Jir {D^k}$. As each element $\vec x$ of $D^k$ is the join of some elements of $U_1\cup\dots\cup U_k$, we have that $\Jir{D^k}=U_1\cup\dots\cup U_k\cong kU$. 

To prove that $\gmin{D^k}\geq \Asp(\Jir D,k)$, let $n:=\gmin{D^k}$ and pick an $n$-element generating set $\set{g_1,\dots,g_n}$ of $D^k$. 
 By \eqref{eq:smZstmrCvr}, we need to show that $k\leq \Sp(U,n)$. So, we need to embed $kU$ into $\nPow $. As $\FS n=\FS{x_1,\dots,x_n}$ is embedded into $\nPow $ and $kU\cong\Jir{D^k}$, it suffices to give an order embedding $\Jir{D^k}\to \FS n$. 
In the \emph{meet-semilattice reduct} $(D^k;\wedge)$ of the lattice $(D^k;\wedge,\vee)$, let $B:=\meetgen{g_1,\dots,g_n}$ denote the meet-subsemilattice generated by $\set{g_1,\dots,g_n}$. 
By the distributivity of the \emph{lattice} $D^k$, each $u\in\Jir{D^k}$ is obtained so that we apply a disjunctive normal form to the generators  $g_1,\dots,g_n$. That is, $u$ is the join of some meets of the generators. By the join-irreducibility of $u$, the join is superfluous, and so $u$ is the meet of some of the $g_1,\dots, g_n$. Thus, $u\in B$, and we have seen that $\Jir{D^k}\subseteq B$.
By the freeness of $\FS n$, there exists a (unique) meet homomorphism 
$\phi\colon \FS n \to B$ such that $\phi(x_i)=g_i$ for all $i\in\set{1,\dots,n}$. Since each of the generators $g_i$ of $B$ is a $\phi$-image, $\phi$ is surjective. Define a function $\psi\colon B\to \FS n$ by the rule $\psi(b):=\bigwedge\set{p\in \FS n: \phi(p)=b}$. Then, for every $b\in B$,  $\phi(\psi(b))= \phi\bigl(\bigwedge\set{p\in \FS n: \phi(p)=b}\bigr) = \bigwedge\set{\phi(p)\in \FS n: \phi(p)=b}=b$ shows that $\phi(\psi(b))=b$. Hence,  $\psi(b)$ is the least preimage of $b$ with respect to $\phi$.  
Now assume that $b_1, b_2\in B$.
If $b_1\leq b_2$, then $\phi(\psi(b_1)\wedge \psi(b_2))=\phi(\psi(b_1))\wedge \phi(\psi(b_2))=b_1\wedge b_2=b_1$ shows that $\psi(b_1)\wedge \psi(b_2)$ is a $\phi$-preimage of $b_1$. As $\psi(b_1)$ is the smallest preimage, we obtain that $\psi(b_1)\leq \psi(b_1)\wedge \psi(b_2)\leq \psi(b_2)$, that is, $\psi$ is order-preserving.
Conversely, if $\psi(b_1)\leq \psi(b_2)$, then $b_1=\phi(\psi(b_1))=\phi(\psi(b_1)\wedge \psi(b_2))= \phi(\psi(b_1))\wedge \phi(\psi(b_2))=b_1\wedge b_2\leq b_2$, whereby $\psi\colon B\to \FS n$ is an order-embedding. Restricting $\psi$ to $\Jir{D^k}$, we obtain an embedding of $\Jir{D^k}$ into $\FS n$, as required. Consequently, $\gmin{D^k}\geq \Asp(\Jir D,k)$.

To prove the converse inequality,  $\gmin{D^k}\leq \Asp(\Jir D,k)$, now we change the meaning of $n$ as follows:  let $n:=\Asp(\Jir D,k)$. We have to show that $D^k$ has an at most $n$-element generating set. Let $U:=\Jir D$; then $kU\cong \Jir{D^k}$ as in the first part of the proof. Furthermore, we know from \eqref{eq:rtrkfjnhTkzkc} that $kU$ is order embedded in $\nPow $. Since $k\geq 2$, $kU$ has no largest element. Thus, using that $\FS n$ is order isomorphic to $\nPow\setminus\set{[n]}$, 
$kU$ is also embedded in $\FS n=\FS{x_1,\dots,x_n}$. So we assume that $kU$ is a subposet of $\FS n$. 
A subset $X$ of $kU$ is called a \emph{down-set} of $kU$ if for every $y\in X$, $\pideal{kU} y\subseteq X$. The collection $\Dn{kU}=(\Dn{kU};\subseteq)$ of all down-sets of $kU$ is a distributive lattice.
Since $kU\cong \Jir{D^k}$, we obtain by the well-known structure theorem of finite distributive lattices, see  Gr\"atzer \cite[Theorem 107]{GLTFound} for example, that $\Dn{kU}\cong D^k$. Hence, it suffices to find an (at most) $n$-element generating set of $\Dn{kU}$.
For $i\in\set{1,\dots,n}$, define $Y_i:=\set{y\in kU: y\leq x_i \text{,  understood in }\FS n}$. Then $Y_i\in\Dn{kU}$, and we are going to show that $\set{Y_1,\dots,Y_n}$ generates $\Dn{kU}$. For every $X\in\Dn{kU}$, $X=\bigcup\set{\pideal{kU} y: y\in X}=\bigvee\set{\pideal{kU} y: y\in X}$.
Therefore (since the meet in $\Dn{kU}$ is the intersection), it suffices to show that for each $u\in kU$, $\pideal{kU}y=\bigcap\set{Y_i: u\in Y_i}$. 
The ``$\subseteq$'' inclusion here is trivial since the $Y_i$'s are down-sets. 
To verify the converse inclusion, assume that $v\in \bigcap\set{Y_i: u\in Y_i}$. This means that for all $i\in\set{1,\dots,n}$, if $u\in Y_i$, then $v\in Y_i$. In other words, for all $i\in\set{1,\dots,n}$, if $u\leq x_i$, then $v\leq x_i$. Thus, $v\leq \bigwedge\set{x_i: u\leq x_i}$. As each element of $\FS n$ is the meet of all elements above itself, $u=\bigwedge\set{x_i: u\leq x_i}$. By this equality and the  just-obtained inequality,  $v\leq u$, that is, $v\in\pideal{kU} u$.
This shows the ``$\supseteq$'' inclusion and completes the proof.
\end{proof}

\section{A Sperner type theorem}\label{sect:exact}
Let us repeat that a poset $U$ is  \emph{bounded} if $0=0_U\in U$ and $1=1_U\in U$. Even though we have not seen the following statement in the literature, all the tools needed in its proof are present in Lubell \cite{lubell},  Griggs, Stahl, and Trotter \cite{griggsatall}, and Dove and Griggs \cite{dovegriggs}; this is why we call it an observation rather than a theorem.

\begin{observation}\label{obs:nsTgtfblS} Let $U$ be a finite poset, let $n,k\in\Nnul$, and let $p:=\Asp(U,1)$,  that is, $p=\min\set{p'\in\Nnul: U\embeddable \Pow{[p']}}$. Then the following assertions hold.

\textup{(a)} If $n\geq p$, then  $\Sp(U,n)\geq \fsb(n-p)$.

\textup{(b)} If $k\geq 1$, then  $\Asp(U,k)\leq p + \afsb(k)$.

\textup{(c)} If $U$ is a bounded  and $n\geq p$, then $\Sp(U,n) =  \fsb(n-p)$, that is, $\Sp(U,n)=\ibinom {n-p}{\lint{(n-p)/2}}$. 

\textup{(d)} If $U$ is bounded and $k\geq 1$, then $\Asp(U,k) = p + \afsb(k)$.
\end{observation}

If $|U|=1$, then $p=0$. Hence, Sperner's Theorem, see \cite{sperner}, is a particular case of Theorem \ref{obs:nsTgtfblS}. Clearly, so is \eqref{eq:griggs}, which we quoted from Griggs, Stahl and Trotter \cite{griggsatall}. The forthcoming Table  \ref{tableWsmall}  shows  that  parts (c) and (d) would fail without assuming that $U$ is bounded.

\begin{proof} As we have already mentioned, all the ideas are taken from  Lubell \cite{lubell},  Griggs, Stahl, and Trotter \cite{griggsatall}, and Dove and Griggs \cite{dovegriggs}. 
To prove part \textup{(a)}, let  $B:=\set{n-p+1, n-p+2, \dots, n}$. As $|B|=p$ and we can replace $U$ with a poset isomorphic to it, we  assume that $U\subseteq \Pow{B}$. 
The  $\lint{(n-p)/2}$-element subsets of $\set{1,\dots,n-p}$ form a $k:=\fsb(n-p)$-element antichain $\Phi$ in $\Pow{[n-p]}$.
For $X_1,X_2\in \Phi$ and $Y_1,Y_2\in U$, if $X_1\neq X_2$, then some $i\in \set{1,\dots, n-p}$ is in $X_1\setminus X_2$ and so $i\in (X_1\cup Y_1)\setminus (X_2\cup Y_2)$. Hence,  $(\set{X\cup Y: X\in \Phi\text{ and }Y\in U};\subseteq) \cong (kU;\leq)$ is a subposet of $\nPow$. Thus, $\Sp(U,n)\geq k = \fsb(n-p)$, as required.  

To prove part \textup{(b)}, observe that for $k\geq 1$, part \textup{(a)} implies that
\begin{equation}
\set{n: p\leq n\in\Nnul\text{ and }k\leq \Sp(U,n)}\supseteq \set{n: p\leq n\in\Nnul\text{ and }k\leq \fsb(n-p)}.
\label{eq:kzLvlshRLGrtn}
\end{equation} 
Observe also that, by \eqref{eq:rtrkfjnhTkzkc}, $k\leq \Sp(U,n)\iff kU\embeddable \nPow$. Hence, we can compute as follows; note that \eqref{eq:kzLvlshRLGrtn} will be used only once.
\begin{align*}
\Asp(U,k)  \reqref{eq:rtrkfjnhTkzkc}
&\min\set{n: n\in\Nnul\text{ and }k\leq \Sp(U,n)} \cr
\overset{k\geq 1}=
&\min\set{n: p\leq n\in\Nnul\text{ and }k\leq \Sp(U,n)}\cr
\rleqref{eq:kzLvlshRLGrtn}
&\min\set{n: p\leq n\in\Nnul\text{ and }k\leq \fsb(n-p)}\cr
=&\min\set{p+n': n'\in\Nnul\text{ and }k\leq \fsb(n')}\cr
=p+{}&\min\set{n': n'\in\Nnul\text{ and }k\leq \fsb(n')}= p+\afsb(k),
\end{align*}
proving part \textup{(b)}.

To prove \textup{(c)}, assume that $U$ is bounded. It suffices to verify that $\Sp(U,n)\leq \fsb(n-p)$, which is the converse of the inequality proved for part \textup{(a)}.
With the notation $k:=\Sp(U,n)$, we know that there exists an order embedding $f\colon kU\to \nPow$. 
Let $U_1$, \dots, $U_k$ be the pairwise disjoint isomorphic copies of $U$
such that $kU$ is the union of them. For $i\in[k]$, denote the restriction of $f$ to $U_i$ by $f_i$, and  let $X_i:=f_i(1_{U_i})$ and $Z_i:=f_i(0_{U_i})$. Since the interval $[Z_i,X_i]=\set{Y\in \nPow: Z_i\subseteq Y\subseteq X_i}$ is order isomorphic to $\Pow{X_i\setminus Z_i}$, it follows that $|X_i\setminus Z_i|\geq p$.  
Hence, we can pick a chain $Z_i=Y^{(i)}_0\subset Y^{(i)}_1\subset\dots\subset Y^{(i)}_{p-1}\subset Y^{(i)}_p=X_i$. 
If we had that $Y^{(i)}_{s} \subseteq Y^{(j)}_t$ for some $i\neq j\in[k]$ and $s,t\in \set{0,\dots, p}$, then 
\begin{align*}
f(0_{U_i})&=
f_i(0_{U_i})= Z_i=Y^{(i)}_0\subseteq 
Y^{(i)}_s
\cr
&\subseteq Y^{(j)}_t  \subseteq  
Y^{(j)}_p=X_j=f_j(1_{U_j})=f(1_{U_j})
\end{align*}
and the fact that $f$ is an order embedding would imply that $0_{U_i}\leq 1_{U_j}$, which is a contradiction. Hence $Y^{(i)}_{s}$ and  $Y^{(j)}_t$ are incomparable for $i\neq j$. Therefore, 
letting $k\chain p=\bigcup_{i\in[k]}\set{y^{(i)}_0, y^{(i)}_1,\dots, y^{(i)}_p}$ with 
$y^{(i)}_0\prec y^{(i)}_1\prec\dots \prec y^{(i)}_p$,
the ``capitalizing map'' $k\chain p\to \nPow$ defined by $y^{(i)}_s\mapsto Y^{(i)}_s$ is an order embedding. Thus, it follows from Griggs, Stahl, and Trotter's result, quoted here in \eqref{eq:griggs}, that $\Sp(U,n)=k\leq \Sp(\chain p,n) =\fsb(n-p)$, as required. We have shown part \textup{(c)}.

To prove part \textup{(d)}, observe that in the argument for  \textup{(b)}, part \textup{(a)} yielded inequality \eqref{obs:nsTgtfblS}, which was used only once in the multi-line computation. Now that part \textup{(c)} turns  \eqref{obs:nsTgtfblS} into an equality, the  multi-line computation turns into a computation proving the required equality $\Asp(U,k)=p+\afsb(k)$, completing the proof.
\end{proof}

\section{Lower and upper estimates for non-bounded posets}\label{sect:loweruppergeneral}
For \emph{any} finite poset $U$, Dove and Griggs \cite{dovegriggs} and Katona and Nagy \cite{KatonaNagy}, independently from each other, 
gave lower estimates and upper estimates of $\Sp(U,n)$. Their estimates are asymptotically equal if $n$ tends to infinity.
Thus, $\Sp(U,n)$ is asymptotically known\footnote{When writing \href{https://arxiv.org/abs/2308.15625v2}{arXiv:2308.15625v2}, the earlier version of this paper, I did not know about Dove and Griggs \cite{dovegriggs} and Katona and Nagy \cite{KatonaNagy}; thank goes to D\'aniel Nagy (the second author of \cite{KatonaNagy})  to call my attention to these two papers.}{} for each $U$. In general, however, this  knowledge   does not give us too much information on $\Sp(U,n)$ for a \emph{small}  $n$.  By parsing the arguments in Dove and Griggs \cite{dovegriggs} or Katona and Nagy \cite{KatonaNagy}, one can obtain some estimates for a  small $n$ but sometimes, putting generality aside,  
other constructions could be easier and could give better estimates. This will be exemplified by two small \emph{concrete} posets; see Propositions \ref{prop:Wposet} and \ref{prop:Vposet}  later.  But first of all, let us agree that the set of all permutations of $[n]$ are denoted by $\Symn$; its members are written in the form $\vp=(\pi_1,\dots,\pi_n)$. For $\vp\in \Symn$, $j\in[n]$ and $X\in\nPow\setminus\set{\emptyset}$, we denote by 
\begin{align}
\iset(j,\vp):=\set{\pi_m: 1\leq m\leq j}, \text{  } \lpos(X,\vp):=\max\set{m\in [n]: \pi_m\in X},
\label{eq:sMhchSb}\\
\text{and }\,\Gamma(X):=\set{\vp\in\Symn:  \iset(\lpos(Z_i,\vp),\vp)\subseteq X}
\label{eq:jlhsvZrngLp}
\end{align}
the \emph{$j$-th \tbf initial \tbf set} of $\vp$,  the \emph{\tbf last \tbf position} of $X$ in $\vp$, and the \emph{set of permutations associated with} $X$,  respectively.
We let $\lpos(\emptyset,\vp):=0$ and  $\iset(0,\vp)=\emptyset$. Of course, we can change ``$\subseteq$'' in \eqref{eq:jlhsvZrngLp} into ``=''. 
The following statement is due to Lubell \cite{lubell} and (apart from terminological changes) was used succesfully by
Dove and Griggs \cite{dovegriggs}, Griggs, Stahl, and Trotter \cite{griggsatall}, and Katona and Nagy \cite{KatonaNagy}:
\begin{align}
&\text{if } X,Y\in\nPow\text{ such that } X\parallel Y,  \text{  then }  \Gamma(X)\cap \Gamma(Y)=\emptyset \text{ and}
\label{eq:rVncBjtRstvn}\\
&\text{for every }X\in\nPow, \text{ we have that }|\Gamma(X)|=|X|! \cdot (n-|X|)!.
\label{eq:mmdnHnRgkzgfclHsg}
\end{align}

\onlymiktex{
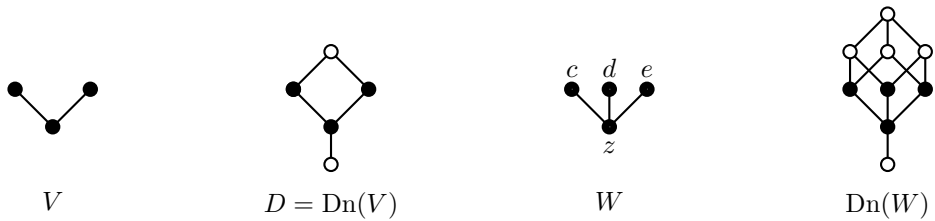
\begin{figure}[hbt]
\begin{tikzpicture}[scale=0.5]
\draw (0,2) coordinate (A);
\draw (1,1) coordinate (B);
\draw (1,0) coordinate (b);
\draw (2,2) coordinate (C);
\draw[thick](A)--(B); \draw[thick](C)--(B); 
\bcirc A; \bcirc C; \bcirc B;
\nodetxt b{$V$}; 
\draw (\dst+5,0) coordinate (D);
\draw (\dst+5,1) coordinate (E);
\draw (\dst+4,2) coordinate (F);
\draw (\dst+6,2) coordinate (G);
\draw (\dst+5,3) coordinate (H);
\draw[thick](D)--(E); \draw[thick](E)--(F); \draw[thick](E)--(G); 
\draw[thick](H)--(F); \draw[thick](H)--(G); 
\bcirc E; \bcirc F; \bcirc F; \bcirc G; \wcirc H;
\nodetxt D{$D=\Dn V$}; \wcirc D;
\draw(2*\dst+8,2) coordinate (K);
\draw (2*\dst+9,2) coordinate (L);
\draw (2*\dst+10,2) coordinate (M);
\draw (2*\dst+9,1) coordinate (J);
\draw (2*\dst+9,0) coordinate (j);
\draw[thick] (J)--(K); \draw[thick](J)--(L); \draw[thick](J)--(M); 
\bcirc K; \bcirc L; \bcirc M;
\bcirc J;
\nodetxt j{$W$};  \nodeuutxt K{$c$}; \nodeuutxt L{$d$}; \nodeuutxt M{$e$};
\draw [white,fill] (J) circle [radius=0.0pt] node [black,below=2pt] {$z$};
\draw (3*\dst+13,0) coordinate (N);
\draw (3*\dst+13,1) coordinate (O);
\draw (3*\dst+12,2) coordinate (P);
\draw (3*\dst+13,2) coordinate (Q);
\draw (3*\dst+14,2) coordinate (R);
\draw (3*\dst+12,3) coordinate (S);
\draw (3*\dst+13,3) coordinate (T);
\draw (3*\dst+14,3) coordinate (U);
\draw (3*\dst+13,4) coordinate (V);
\draw[thick](N)--(O); \draw[thick](O)--(P);  \draw[thick](O)--(R);   \draw[thick](O)--(Q);  \draw[thick](P)--(S); 
\draw[thick](P)--(T);  \draw[thick](Q)--(S);  \draw[thick](Q)--(U);  \draw[thick](R)--(T); \draw[thick](R)--(U); \draw[thick](S)--(V); \draw[thick](T)--(V); \draw[thick](U)--(V);
\bcirc O; \bcirc P; \bcirc Q; \bcirc R;
\wcirc S; \wcirc T; \wcirc U; \wcirc V; 
\nodetxt N{$\Dn W$};
\wcirc N; 
\end{tikzpicture}
\caption{Two posets and the corresponding distributive lattices}
\label{figone}
\end{figure}
}

Next, let $W$ denote  the 4-element poset $W$ with  0 and three maximal elements, see Figure \ref{figone}.
For  $n\in\Nplu$ we define  
\begin{align}
\upS(W,n)&:=\Bigl\lfloor\frac{n}
{3n-2-2\lint{n/2}}\cdot \fsb(n-1) \Bigr\rfloor.
\label{eq:cjWhlmgTlP}
\end{align}
and, with the convention that $\ibinom {n_1}{n_2}=0$  unless $0\leq n_2\leq n_1$, let
\begin{equation}
\loS(W,n):=
\begin{cases}\displaystyle{
\sum_{i=0}^{\lint{n/3}-1}  \sum_{j=0}^{i} 3^j \binom{i}{j} \binom{n-3i-3}{\lint{(n-1)/2}+j-3i}
}&\text{if }n\notin\set{3,5,7},\cr
\displaystyle{
\sum_{i=0}^{\lint{n/3}-1}  \sum_{j=0}^{i} 3^j \binom{i}{j}  \binom{n-3i-3}{{(n-3)/2}+j-3i}
}&\text{if }n\in\set{3,5,7}.
\end{cases}
\label{eq:rtrGtvZTlP}
\end{equation}
Note that $\loS(W,1)=\Sp(W,1)$ and $\loS(W,2)=\Sp(W,2)$. Hence, we can often assume that $n\geq 3$. The \emph{natural density} of a subset $X$ of $\Nplu$ is defined to be $\lim_{n\to\infty}|X\cap [n]|/n$, provided that this limit exists.

\begin{proposition}\label{prop:Wposet}
For $3\leq n\in\Nplu$,  $\upS(W,n)$  and $\loS(W,n)$ defined in \eqref{eq:cjWhlmgTlP}   and \eqref{eq:rtrGtvZTlP} are an upper estimate and a lower estimate of $\Sp(W,n)$, that is, $\loS(W,n)\leq \Sp(W,n)\leq \upS(W,n)$. The functions $\loS(W,-)$, $\Sp(W,-)$, $\upS(W,-)$, and $\frac 1 4\cdot \fsb(-)$ are asymptotically equal.  Furthermore, denoting the left adjoints of the functions $\loS(W,-)$ and $\upS(W,-)$ by $\loS^\ast(W,-)$ and $\upS^\ast(W,-)$, respectively, 
\begin{equation}
\upS^\ast(W,k)\leq \Asp(W,k) \leq \loS^\ast(W,k)\,\,\text{ and }\,\,0\leq \loS^\ast(W,k)-\upS^\ast(W,k)\leq 1
\label{eq:mLptpsGPs}
\end{equation}
for all $k\in\Nplu$ and, moreover, 
the natural density of the set $\set{k\in \Nplu: \upS^\ast(W,k)=\loS^\ast(W,k)}$ is $1$. 
\end{proposition}

The proof below uses lots from the proofs in Dove and Griggs \cite{dovegriggs} and Katona and Nagy \cite{KatonaNagy}; we are going to discuss the differences in Remark \ref{ref:diffDGKN}.

\begin{proof} First, we deal with $\upS(W,n)$.
Let $k:=\Sp(W,n)$, and let $W_1,\dots,W_k$ be pairwise unrelated copies of $W$ in $\nPow$. In particular, $(W_i,\subseteq)$ is order isomorphic to $W$. 
The assumption $n\geq 3$ gives that $\upS(W,n)\geq 1$. Thus, we can assume that $k\geq 2$ as otherwise $\Sp(W,n)=k\leq \upS(W,n)$ is obvious. 
In accordance with Figure \ref{figone}, we use the notation  $W_i=\set{Z_i,C_i,D_i,E_i}$ where $Z_i\subset C_i$, $C_i\parallel D_i$, etc., and $Z_i\parallel E_j$ for $i\neq  j$, etc. As it is trivial (and used also in  Dove and Griggs \cite{dovegriggs} and Katona and Nagy \cite{KatonaNagy}), if $i\neq j\in[k]$, $Y\in \nPow$, 
$Y',Y''\in W_i$, and $Y'\subseteq Y\subseteq Y''$, then 
$W_i\cup\set{Y}$ is still unrelated to $W_j$; we are going to use this ``convexity principle'' implicitly. As its first use, we can assume that $Z_i$ equals the intersection $C_i\cap D_i\cap E_i$ as otherwise we could replace $Z_i$ by this intersection.

We claim that with some pairwise distinct elements  $c_i,d_i,e_i\in[n]\setminus Z_i$, we can change $W_i$ to 
 $W_i'=\set{Z_i, Z_i\cup\set{c_i},  Z_i\cup\set{d_i},  Z_i\cup\set{e_i}}$ such that $W_1$, \dots, $W_{i-1}$, $W_i'$, $W_{i+1}$, \dots, $W_k$ 
still form a system of pairwise unrelated copies of $W$.
Let $C_i'=C_i\setminus  Z_i$, $D_i'=D_i\setminus  Z_i$,  and $E_i'=E_i\setminus  Z_i$. 
 If at least one of $C_i'$, $D_i'$ and $E_i'$ is not a subset of the union of the other two, say, 
 $C'_i\nsubseteq D'_i\cup E'_i$, then any choice of $c_i\in C_i'\setminus(D_i'\cup E_i')$,  $d_i\in D_i'\setminus E_i'$, and  $e_i\in E_i'\setminus D_i'$ does the job by the convexity principle.  So we can assume that each of $C_i'$, $D_i'$ and $E_i'$ is a subset of the union of the other two. Take an element from $C'_i\setminus D'_i$. As $C'_i\subseteq D'_i\cup E'_i$, this element is in $E'_i$; we denote it by  $x_{C,\neg D,E}$ .  The meaning of its  subscripts is that  $x_{C,\neg D,E}$ belongs to $C'_i$ and $E'_i$ but not to $D'_i$. By symmetry, we obtain elements  $x_{C, D,\neg E}\in (C_i\cap D_i)\setminus E_i$  and  $x_{\neg C, D,E}\in (D_i\cap E_i)\setminus C_i$.  The subscripts show that these three elements are pairwise distinct. This fact and the convexity principle imply that $W_i$ can be changed to the required form with $c_i:=x_{C,\neg D,E}$, $d_i:= x_{C, D,\neg E}$, and $e_i:=x_{\neg C, D,E}$. 
 Therefore, in the rest of the proof, we assume that for all $i\in[k]$, 
  \begin{equation}
W_i=\set{Z_i, Z_i\cup\set{c_i},  Z_i\cup\set{d_i},  Z_i\cup\set {e_i}}.
\label{eq:kvStrSgkZlsM}
\end{equation} 

Letting $\Gamma_i:=\Gamma(Z_i)\cup \Gamma(Z_i\cup\set{c_i}) \cup \Gamma(Z_i\cup\set{d_i})
\cup \Gamma(Z_i\cup\set{e_i})$, our next task is to find a reasonable lower bound on $|\Gamma_i|$. 
With the notation $z_i:=|Z_i|$, we can order the first $z_i$ compontens of a $\vp=(\pi_1,\dots,\pi_n)\in \Gamma(Z_i)\cap \Gamma(Z_i\cup \set{c_i})$, which form the set $Z_i$, in  $z!$ ways. We have that $\pi_{z_i+1}=c_i$, and the last $n-z_i-1$ components can be ordered in 
 $(n-z_i-1)!$ ways. Hence,  $|\Gamma(Z_i\cup \set{c_i})|= z_i!(n-z_i-1)!$, and the same is true for 
 $|\Gamma(Z_i\cup \set{d_i})|$ and  $|\Gamma(Z_i\cup \set{e_i})|$. This fact, 
\eqref{eq:rVncBjtRstvn}, \eqref{eq:mmdnHnRgkzgfclHsg}, and the inclusion–exclusion principle yield that 
\begin{equation}
\begin{aligned}
|\Gamma_i|&=g_0(z_i) \text{ where } g_0(x):=x!(n-x)! \cr
&+ 3(x+1)!(n-x-1)!  -3x!(n-x-1)! \cr
&= (n+2x_i)x_i!(n-1-x_i)!
\end{aligned}
\label{eq:slBsjmGsTlt}
\end{equation}
Note that $z_i\geq 1$ as otherwise $Z_i=\emptyset$ would be comparable with $Z_j$ for $j\in [k]\setminus \set i$. (Here we used that $k\geq 2$.) We also have that $z_i\leq n-1$ since $Z_i\cup\set{c_i}\in\nPow$. 
Thus, we can use later that $x\in [n-1]=\set{1,\dots,n-1}$.  
For  the auxiliary function $g_1(x):=g_0(x)-g_0(x-1)$, we have that 
$g_1(x)=g_2(x)\cdot x!(n-1-x)!$ where $g_2(x)=4x^2-2x-(n^2-2n)$. 
The smaller root of the quadratic equation $g_2(x)=0$ is negative while the larger one is strictly between $n/2-1/2$ and $n/2-1/4$. Hence the largest integer $x$ for which $g_2(x)$ and so $g_1(x)$ are negative is $x=\lint{(n-1)/2}$. Therefore, on the set $[n-1]$,  $g_0$ takes its minimum at $\lint{(n-1)/2}$. Let  $M:=g_0(\lint{(n-1)/2})$. Using that the $\Gamma_i$'s are pairwise disjoint by  \eqref{eq:rVncBjtRstvn} and  $\Gamma_1\cup\dots \Gamma_k\subseteq \Symn$, we obtain that 
\begin{equation}
kM=\sum_{i=1}^k M\leq \sum_{i=1}^|\Gamma_i|\leq |\Symn|=n!\, .
\label{eq:mTlThlrKtmklmSrrnd}
\end{equation}
Dividing this inequality by $M$ (and dealing with odd $n$'s and even $n$'-s separately), we obtain the required inequality $\Sp(W,n)=k\leq \upS(W,n)$.

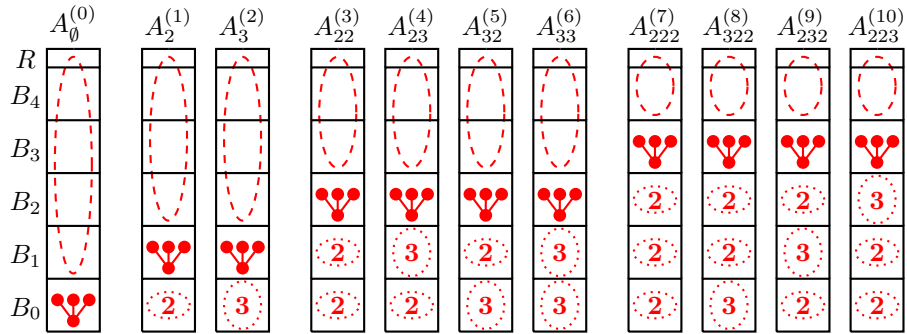
\begin{figure}[hbt]
\begin{tikzpicture}[scale=1.0]
\draw(0.6*\hot,\vot+0*\vot) coordinate (L0);
\draw(0.6*\hot,\vot+1*\het) coordinate (L1);
\draw(0.6*\hot,\vot+2*\vot) coordinate (L2);
\draw(0.6*\hot,\vot+3*\vot) coordinate (L3);
\draw(0.6*\hot,\vot+4*\vot) coordinate (L4);
\draw(0.6*\hot,\vot+4.77*\vot) coordinate (L5);
\nodeuutxt{L0}{$B_0$};
\nodeuutxt{L1}{$B_1$};
\nodeuutxt{L2}{$B_2$};
\nodeuutxt{L3}{$B_3$};
\nodeuutxt{L4}{$B_4$};
\nodeuutxt{L5}{$R$};
\draw(\hot,\vot) coordinate (a);
\draw(\hot+\wit,\vot) coordinate (b);
\draw(\hot,\vot+\het) coordinate (c);
\draw(\hot+\wit,\vot+\het) coordinate (d);
\draw(\hot,\vot+2*\het) coordinate (e);
\draw(\hot+\wit,\vot+2*\het) coordinate (f);
\draw(\hot,\vot+3*\het) coordinate (g);
\draw(\hot+\wit,\vot+3*\het) coordinate (h);
\draw(\hot,\vot+4*\het) coordinate (i);
\draw(\hot+\wit,\vot+4*\het) coordinate (j);
\draw(\hot,\vot+5*\het) coordinate (k);
\draw(\hot+\wit,\vot+5*\het) coordinate (l);
\draw(\hot,\vot+5.35*\het) coordinate (m);
\draw(\hot+0.5*\wit,\vot+5.35*\het) coordinate (s);
\draw(\hot+\wit,\vot+5.35*\het) coordinate (n);
\draw(\hot+0.5*\wit,\vot+0.2*\het) coordinate (o);
\draw(\hot+0.5*\wit,\vot+0.6*\het) coordinate (q);
\draw(\hot+0.2*\wit,\vot+0.6*\het) coordinate (p);
\draw(\hot+0.8*\wit,\vot+0.6*\het) coordinate (r);
\draw(\hot+0.5*\wit,\vot+3.15*\het) coordinate (ell);
\draw[thick,red]  (o)--(p);\draw[thick,red]  (o)--(q);
\draw[thick,red]  (o)--(r);
\draw[thick,red,dashed] (ell) ellipse (7pt and 41pt);
\draw[thick] (a)--(b)--(d)--(c)--(a);
\draw[thick]  (c)--(e)--(f)--(d);
\draw[thick]  (e)--(g)--(h)--(f);
\draw[thick]  (g)--(i)--(j)--(h);
\draw[thick]  (i)--(k)--(l)--(j);
\draw[thick]  (k)--(m)--(n)--(l);
\lcirc {o};\lcirc{q};\lcirc{p};\lcirc{r};
\nodeutxt{s}{$\fbx0{\emptyset}$};
%
%
\draw(1*\wit+2*\gap+\hot,\vot) coordinate (a);
\draw(1*\wit+2*\gap+\hot+\wit,\vot) coordinate (b);
\draw(1*\wit+2*\gap+\hot,\vot+\het) coordinate (c);
\draw(1*\wit+2*\gap+\hot+\wit,\vot+\het) coordinate (d);
\draw(1*\wit+2*\gap+\hot,\vot+2*\het) coordinate (e);
\draw(1*\wit+2*\gap+\hot+\wit,\vot+2*\het) coordinate (f);
\draw(1*\wit+2*\gap+\hot,\vot+3*\het) coordinate (g);
\draw(1*\wit+2*\gap+\hot+\wit,\vot+3*\het) coordinate (h);
\draw(1*\wit+2*\gap+\hot,\vot+4*\het) coordinate (i);
\draw(1*\wit+2*\gap+\hot+\wit,\vot+4*\het) coordinate (j);
\draw(1*\wit+2*\gap+\hot,\vot+5*\het) coordinate (k);
\draw(1*\wit+2*\gap+\hot+\wit,\vot+5*\het) coordinate (l);
\draw(1*\wit+2*\gap+\hot,\vot+5.35*\het) coordinate (m);
\draw(1*\wit+2*\gap+\hot+0.5*\wit,\vot+5.35*\het) coordinate (s);
\draw(1*\wit+2*\gap+\hot+\wit,\vot+5.35*\het) coordinate (n);
\draw(1*\wit+2*\gap+\hot+0.5*\wit,\vot+1.2*\het) coordinate (o);
\draw(1*\wit+2*\gap+\hot+0.5*\wit,\vot+1.6*\het) coordinate (q);
\draw(1*\wit+2*\gap+\hot+0.2*\wit,\vot+1.6*\het) coordinate (p);
\draw(1*\wit+2*\gap+\hot+0.8*\wit,\vot+1.6*\het) coordinate (r);
\draw(1*\wit+2*\gap+\hot+0.5*\wit,\vot+3.65*\het) coordinate (ell);
\draw(1*\wit+2*\gap+\hot+0.5*\wit,\vot+0.5*\het) coordinate (cp0);
\draw(1*\wit+2*\gap+\hot+0.5*\wit,\vot+1.5*\het) coordinate (cp1);
\draw(1*\wit+2*\gap+\hot+0.5*\wit,\vot+2.5*\het) coordinate (cp2);
\draw[thick,red]  (o)--(p);\draw[thick,red]  (o)--(q);
\draw[thick,red]  (o)--(r);
\draw[thick,red,dashed] (ell) ellipse (7pt and 31pt);
\draw[thick] (a)--(b)--(d)--(c)--(a);
\draw[thick]  (c)--(e)--(f)--(d);
\draw[thick]  (e)--(g)--(h)--(f);
\draw[thick]  (g)--(i)--(j)--(h);
\draw[thick]  (i)--(k)--(l)--(j);
\draw[thick]  (k)--(m)--(n)--(l);
\lcirc {o};\lcirc{q};\lcirc{p};\lcirc{r};
\nodeutxt{s}{$\fbx1{2}$};
\draw[thick,red,dotted] (cp0) ellipse (8pt and 5pt); \nodettxt{cp0}2;
%
\draw(2*\wit+3*\gap+\hot,\vot) coordinate (a);
\draw(2*\wit+3*\gap+\hot+\wit,\vot) coordinate (b);
\draw(2*\wit+3*\gap+\hot,\vot+\het) coordinate (c);
\draw(2*\wit+3*\gap+\hot+\wit,\vot+\het) coordinate (d);
\draw(2*\wit+3*\gap+\hot,\vot+2*\het) coordinate (e);
\draw(2*\wit+3*\gap+\hot+\wit,\vot+2*\het) coordinate (f);
\draw(2*\wit+3*\gap+\hot,\vot+3*\het) coordinate (g);
\draw(2*\wit+3*\gap+\hot+\wit,\vot+3*\het) coordinate (h);
\draw(2*\wit+3*\gap+\hot,\vot+4*\het) coordinate (i);
\draw(2*\wit+3*\gap+\hot+\wit,\vot+4*\het) coordinate (j);
\draw(2*\wit+3*\gap+\hot,\vot+5*\het) coordinate (k);
\draw(2*\wit+3*\gap+\hot+\wit,\vot+5*\het) coordinate (l);
\draw(2*\wit+3*\gap+\hot,\vot+5.35*\het) coordinate (m);
\draw(2*\wit+3*\gap+\hot+0.5*\wit,\vot+5.35*\het) coordinate (s);
\draw(2*\wit+3*\gap+\hot+\wit,\vot+5.35*\het) coordinate (n);
\draw(2*\wit+3*\gap+\hot+0.5*\wit,\vot+1.2*\het) coordinate (o);
\draw(2*\wit+3*\gap+\hot+0.5*\wit,\vot+1.6*\het) coordinate (q);
\draw(2*\wit+3*\gap+\hot+0.2*\wit,\vot+1.6*\het) coordinate (p);
\draw(2*\wit+3*\gap+\hot+0.8*\wit,\vot+1.6*\het) coordinate (r);
\draw(2*\wit+3*\gap+\hot+0.5*\wit,\vot+3.65*\het) coordinate (ell);
\draw(2*\wit+3*\gap+\hot+0.5*\wit,\vot+0.5*\het) coordinate (cp0);
\draw(2*\wit+3*\gap+\hot+0.5*\wit,\vot+1.5*\het) coordinate (cp1);
\draw(2*\wit+3*\gap+\hot+0.5*\wit,\vot+2.5*\het) coordinate (cp2);
\draw[thick,red]  (o)--(p); \draw[thick,red]  (o)--(q); \draw[thick,red]  (o)--(r);
\draw[thick,red,dashed] (ell) ellipse (7pt and 31pt);
\draw[thick] (a)--(b)--(d)--(c)--(a);
\draw[thick]  (c)--(e)--(f)--(d);
\draw[thick]  (e)--(g)--(h)--(f);
\draw[thick]  (g)--(i)--(j)--(h);
\draw[thick]  (i)--(k)--(l)--(j);
\draw[thick]  (k)--(m)--(n)--(l);
\lcirc {o};\lcirc{q};\lcirc{p};\lcirc{r};
\nodeutxt{s}{$\fbx2{3}$};
\draw[thick,red,dotted] (cp0) ellipse (7pt and 9pt); \nodettxt{cp0}3;
%
%
\draw(3*\wit+5*\gap+\hot,\vot) coordinate (a);
\draw(3*\wit+5*\gap+\hot+\wit,\vot) coordinate (b);
\draw(3*\wit+5*\gap+\hot,\vot+\het) coordinate (c);
\draw(3*\wit+5*\gap+\hot+\wit,\vot+\het) coordinate (d);
\draw(3*\wit+5*\gap+\hot,\vot+2*\het) coordinate (e);
\draw(3*\wit+5*\gap+\hot+\wit,\vot+2*\het) coordinate (f);
\draw(3*\wit+5*\gap+\hot,\vot+3*\het) coordinate (g);
\draw(3*\wit+5*\gap+\hot+\wit,\vot+3*\het) coordinate (h);
\draw(3*\wit+5*\gap+\hot,\vot+4*\het) coordinate (i);
\draw(3*\wit+5*\gap+\hot+\wit,\vot+4*\het) coordinate (j);
\draw(3*\wit+5*\gap+\hot,\vot+5*\het) coordinate (k);
\draw(3*\wit+5*\gap+\hot+\wit,\vot+5*\het) coordinate (l);
\draw(3*\wit+5*\gap+\hot,\vot+5.35*\het) coordinate (m);
\draw(3*\wit+5*\gap+\hot+0.5*\wit,\vot+5.35*\het) coordinate (s);
\draw(3*\wit+5*\gap+\hot+\wit,\vot+5.35*\het) coordinate (n);
\draw(3*\wit+5*\gap+\hot+0.5*\wit,\vot+2.2*\het) coordinate (o);
\draw(3*\wit+5*\gap+\hot+0.5*\wit,\vot+2.6*\het) coordinate (q);
\draw(3*\wit+5*\gap+\hot+0.2*\wit,\vot+2.6*\het) coordinate (p);
\draw(3*\wit+5*\gap+\hot+0.8*\wit,\vot+2.6*\het) coordinate (r);
\draw(3*\wit+5*\gap+\hot+0.5*\wit,\vot+4.15*\het) coordinate (ell);
\draw(3*\wit+5*\gap+\hot+0.5*\wit,\vot+0.5*\het) coordinate (cp0);
\draw(3*\wit+5*\gap+\hot+0.5*\wit,\vot+1.5*\het) coordinate (cp1);
\draw(3*\wit+5*\gap+\hot+0.5*\wit,\vot+2.5*\het) coordinate (cp2);
\draw[thick,red]  (o)--(p);\draw[thick,red]  (o)--(q); \draw[thick,red]  (o)--(r);
\draw[thick,red,dashed] (ell) ellipse (7pt and 21pt);
\draw[thick] (a)--(b)--(d)--(c)--(a);
\draw[thick]  (c)--(e)--(f)--(d);
\draw[thick]  (e)--(g)--(h)--(f);
\draw[thick]  (g)--(i)--(j)--(h);
\draw[thick]  (i)--(k)--(l)--(j);
\draw[thick]  (k)--(m)--(n)--(l);
\lcirc {o};\lcirc{q};\lcirc{p};\lcirc{r};
\nodeutxt{s}{$\fbx3{22}$};
\draw[thick,red,dotted] (cp0) ellipse (8pt and 5pt); \nodettxt{cp0}2;
\draw[thick,red,dotted] (cp1) ellipse (8pt and 5pt); \nodettxt{cp1}2;
%
\draw(4*\wit+6*\gap+\hot,\vot) coordinate (a);
\draw(4*\wit+6*\gap+\hot+\wit,\vot) coordinate (b);
\draw(4*\wit+6*\gap+\hot,\vot+\het) coordinate (c);
\draw(4*\wit+6*\gap+\hot+\wit,\vot+\het) coordinate (d);
\draw(4*\wit+6*\gap+\hot,\vot+2*\het) coordinate (e);
\draw(4*\wit+6*\gap+\hot+\wit,\vot+2*\het) coordinate (f);
\draw(4*\wit+6*\gap+\hot,\vot+3*\het) coordinate (g);
\draw(4*\wit+6*\gap+\hot+\wit,\vot+3*\het) coordinate (h);
\draw(4*\wit+6*\gap+\hot,\vot+4*\het) coordinate (i);
\draw(4*\wit+6*\gap+\hot+\wit,\vot+4*\het) coordinate (j);
\draw(4*\wit+6*\gap+\hot,\vot+5*\het) coordinate (k);
\draw(4*\wit+6*\gap+\hot+\wit,\vot+5*\het) coordinate (l);
\draw(4*\wit+6*\gap+\hot,\vot+5.35*\het) coordinate (m);
\draw(4*\wit+6*\gap+\hot+0.5*\wit,\vot+5.35*\het) coordinate (s);
\draw(4*\wit+6*\gap+\hot+\wit,\vot+5.35*\het) coordinate (n);
\draw(4*\wit+6*\gap+\hot+0.5*\wit,\vot+2.2*\het) coordinate (o);
\draw(4*\wit+6*\gap+\hot+0.5*\wit,\vot+2.6*\het) coordinate (q);
\draw(4*\wit+6*\gap+\hot+0.2*\wit,\vot+2.6*\het) coordinate (p);
\draw(4*\wit+6*\gap+\hot+0.8*\wit,\vot+2.6*\het) coordinate (r);
\draw(4*\wit+6*\gap+\hot+0.5*\wit,\vot+4.15*\het) coordinate (ell);
\draw(4*\wit+6*\gap+\hot+0.5*\wit,\vot+0.5*\het) coordinate (cp0);
\draw(4*\wit+6*\gap+\hot+0.5*\wit,\vot+1.5*\het) coordinate (cp1);
\draw(4*\wit+6*\gap+\hot+0.5*\wit,\vot+2.5*\het) coordinate (cp2);
\draw[thick,red]  (o)--(p);\draw[thick,red]  (o)--(q); \draw[thick,red]  (o)--(r);
\draw[thick,red,dashed] (ell) ellipse (7pt and 21pt);
\draw[thick] (a)--(b)--(d)--(c)--(a);
\draw[thick]  (c)--(e)--(f)--(d);
\draw[thick]  (e)--(g)--(h)--(f);
\draw[thick]  (g)--(i)--(j)--(h);
\draw[thick]  (i)--(k)--(l)--(j);
\draw[thick]  (k)--(m)--(n)--(l);
\lcirc {o};\lcirc{q};\lcirc{p};\lcirc{r};
\nodeutxt{s}{$\fbx4{23}$};
\draw[thick,red,dotted] (cp0) ellipse (8pt and 5pt); \nodettxt{cp0}2;
\draw[thick,red,dotted] (cp1) ellipse (7pt and 9pt); \nodettxt{cp1}3;
%
\draw(5*\wit+7*\gap+\hot,\vot) coordinate (a);
\draw(5*\wit+7*\gap+\hot+\wit,\vot) coordinate (b);
\draw(5*\wit+7*\gap+\hot,\vot+\het) coordinate (c);
\draw(5*\wit+7*\gap+\hot+\wit,\vot+\het) coordinate (d);
\draw(5*\wit+7*\gap+\hot,\vot+2*\het) coordinate (e);
\draw(5*\wit+7*\gap+\hot+\wit,\vot+2*\het) coordinate (f);
\draw(5*\wit+7*\gap+\hot,\vot+3*\het) coordinate (g);
\draw(5*\wit+7*\gap+\hot+\wit,\vot+3*\het) coordinate (h);
\draw(5*\wit+7*\gap+\hot,\vot+4*\het) coordinate (i);
\draw(5*\wit+7*\gap+\hot+\wit,\vot+4*\het) coordinate (j);
\draw(5*\wit+7*\gap+\hot,\vot+5*\het) coordinate (k);
\draw(5*\wit+7*\gap+\hot+\wit,\vot+5*\het) coordinate (l);
\draw(5*\wit+7*\gap+\hot,\vot+5.35*\het) coordinate (m);
\draw(5*\wit+7*\gap+\hot+0.5*\wit,\vot+5.35*\het) coordinate (s);
\draw(5*\wit+7*\gap+\hot+\wit,\vot+5.35*\het) coordinate (n);
\draw(5*\wit+7*\gap+\hot+0.5*\wit,\vot+2.2*\het) coordinate (o);
\draw(5*\wit+7*\gap+\hot+0.5*\wit,\vot+2.6*\het) coordinate (q);
\draw(5*\wit+7*\gap+\hot+0.2*\wit,\vot+2.6*\het) coordinate (p);
\draw(5*\wit+7*\gap+\hot+0.8*\wit,\vot+2.6*\het) coordinate (r);
\draw(5*\wit+7*\gap+\hot+0.5*\wit,\vot+4.15*\het) coordinate (ell);
\draw(5*\wit+7*\gap+\hot+0.5*\wit,\vot+0.5*\het) coordinate (cp0);
\draw(5*\wit+7*\gap+\hot+0.5*\wit,\vot+1.5*\het) coordinate (cp1);
\draw(5*\wit+7*\gap+\hot+0.5*\wit,\vot+2.5*\het) coordinate (cp2);
\draw[thick,red]  (o)--(p);\draw[thick,red]  (o)--(q); \draw[thick,red]  (o)--(r);
\draw[thick,red,dashed] (ell) ellipse (7pt and 21pt);
\draw[thick] (a)--(b)--(d)--(c)--(a);
\draw[thick]  (c)--(e)--(f)--(d);
\draw[thick]  (e)--(g)--(h)--(f);
\draw[thick]  (g)--(i)--(j)--(h);
\draw[thick]  (i)--(k)--(l)--(j);
\draw[thick]  (k)--(m)--(n)--(l);
\lcirc {o};\lcirc{q};\lcirc{p};\lcirc{r};
\nodeutxt{s}{$\fbx5{32}$};
\draw[thick,red,dotted] (cp1) ellipse (8pt and 5pt); \nodettxt{cp1}2;
\draw[thick,red,dotted] (cp0) ellipse (7pt and 9pt); \nodettxt{cp0}3;
%
%
\draw(6*\wit+8*\gap+\hot,\vot) coordinate (a);
\draw(6*\wit+8*\gap+\hot+\wit,\vot) coordinate (b);
\draw(6*\wit+8*\gap+\hot,\vot+\het) coordinate (c);
\draw(6*\wit+8*\gap+\hot+\wit,\vot+\het) coordinate (d);
\draw(6*\wit+8*\gap+\hot,\vot+2*\het) coordinate (e);
\draw(6*\wit+8*\gap+\hot+\wit,\vot+2*\het) coordinate (f);
\draw(6*\wit+8*\gap+\hot,\vot+3*\het) coordinate (g);
\draw(6*\wit+8*\gap+\hot+\wit,\vot+3*\het) coordinate (h);
\draw(6*\wit+8*\gap+\hot,\vot+4*\het) coordinate (i);
\draw(6*\wit+8*\gap+\hot+\wit,\vot+4*\het) coordinate (j);
\draw(6*\wit+8*\gap+\hot,\vot+5*\het) coordinate (k);
\draw(6*\wit+8*\gap+\hot+\wit,\vot+5*\het) coordinate (l);
\draw(6*\wit+8*\gap+\hot,\vot+5.35*\het) coordinate (m);
\draw(6*\wit+8*\gap+\hot+0.5*\wit,\vot+5.35*\het) coordinate (s);
\draw(6*\wit+8*\gap+\hot+\wit,\vot+5.35*\het) coordinate (n);
\draw(6*\wit+8*\gap+\hot+0.5*\wit,\vot+2.2*\het) coordinate (o);
\draw(6*\wit+8*\gap+\hot+0.5*\wit,\vot+2.6*\het) coordinate (q);
\draw(6*\wit+8*\gap+\hot+0.2*\wit,\vot+2.6*\het) coordinate (p);
\draw(6*\wit+8*\gap+\hot+0.8*\wit,\vot+2.6*\het) coordinate (r);
\draw(6*\wit+8*\gap+\hot+0.5*\wit,\vot+4.15*\het) coordinate (ell);
\draw(6*\wit+8*\gap+\hot+0.5*\wit,\vot+0.5*\het) coordinate (cp0);
\draw(6*\wit+8*\gap+\hot+0.5*\wit,\vot+1.5*\het) coordinate (cp1);
\draw(6*\wit+8*\gap+\hot+0.5*\wit,\vot+2.5*\het) coordinate (cp2);
\draw[thick,red]  (o)--(p);\draw[thick,red]  (o)--(q); \draw[thick,red]  (o)--(r);
\draw[thick,red,dashed] (ell) ellipse (7pt and 21pt);
\draw[thick] (a)--(b)--(d)--(c)--(a);
\draw[thick]  (c)--(e)--(f)--(d);
\draw[thick]  (e)--(g)--(h)--(f);
\draw[thick]  (g)--(i)--(j)--(h);
\draw[thick]  (i)--(k)--(l)--(j);
\draw[thick]  (k)--(m)--(n)--(l);
\lcirc {o};\lcirc{q};\lcirc{p};\lcirc{r};
\nodeutxt{s}{$\fbx6{33}$};
\draw[thick,red,dotted] (cp0) ellipse (7pt and 9pt); \nodettxt{cp0}3;
\draw[thick,red,dotted] (cp1) ellipse (7pt and 9pt); \nodettxt{cp1}3;
%
\draw(7*\wit+10*\gap+\hot,\vot) coordinate (a);
\draw(7*\wit+10*\gap+\hot+\wit,\vot) coordinate (b);
\draw(7*\wit+10*\gap+\hot,\vot+\het) coordinate (c);
\draw(7*\wit+10*\gap+\hot+\wit,\vot+\het) coordinate (d);
\draw(7*\wit+10*\gap+\hot,\vot+2*\het) coordinate (e);
\draw(7*\wit+10*\gap+\hot+\wit,\vot+2*\het) coordinate (f);
\draw(7*\wit+10*\gap+\hot,\vot+3*\het) coordinate (g);
\draw(7*\wit+10*\gap+\hot+\wit,\vot+3*\het) coordinate (h);
\draw(7*\wit+10*\gap+\hot,\vot+4*\het) coordinate (i);
\draw(7*\wit+10*\gap+\hot+\wit,\vot+4*\het) coordinate (j);
\draw(7*\wit+10*\gap+\hot,\vot+5*\het) coordinate (k);
\draw(7*\wit+10*\gap+\hot+\wit,\vot+5*\het) coordinate (l);
\draw(7*\wit+10*\gap+\hot,\vot+5.35*\het) coordinate (m);
\draw(7*\wit+10*\gap+\hot+0.5*\wit,\vot+5.35*\het) coordinate (s);
\draw(7*\wit+10*\gap+\hot+\wit,\vot+5.35*\het) coordinate (n);
\draw(7*\wit+10*\gap+\hot+0.5*\wit,\vot+3.2*\het) coordinate (o);
\draw(7*\wit+10*\gap+\hot+0.5*\wit,\vot+3.6*\het) coordinate (q);
\draw(7*\wit+10*\gap+\hot+0.2*\wit,\vot+3.6*\het) coordinate (p);
\draw(7*\wit+10*\gap+\hot+0.8*\wit,\vot+3.6*\het) coordinate (r);
\draw(7*\wit+10*\gap+\hot+0.5*\wit,\vot+4.65*\het) coordinate (ell);
\draw(7*\wit+10*\gap+\hot+0.5*\wit,\vot+0.5*\het) coordinate (cp0);
\draw(7*\wit+10*\gap+\hot+0.5*\wit,\vot+1.5*\het) coordinate (cp1);
\draw(7*\wit+10*\gap+\hot+0.5*\wit,\vot+2.5*\het) coordinate (cp2);
\draw[thick,red]  (o)--(p);\draw[thick,red]  (o)--(q); \draw[thick,red]  (o)--(r);
\draw[thick,red,dashed] (ell) ellipse (7pt and 11pt);
\draw[thick] (a)--(b)--(d)--(c)--(a);
\draw[thick]  (c)--(e)--(f)--(d);
\draw[thick]  (e)--(g)--(h)--(f);
\draw[thick]  (g)--(i)--(j)--(h);
\draw[thick]  (i)--(k)--(l)--(j);
\draw[thick]  (k)--(m)--(n)--(l);
\lcirc {o};\lcirc{q};\lcirc{p};\lcirc{r};
\nodeutxt{s}{$\fbx7{222}$};
\draw[thick,red,dotted] (cp0) ellipse (8pt and 5pt); \nodettxt{cp0}2;
\draw[thick,red,dotted] (cp1) ellipse (8pt and 5pt); \nodettxt{cp1}2;
\draw[thick,red,dotted] (cp2) ellipse (8pt and 5pt); \nodettxt{cp2}2;
%
%
\draw(8*\wit+11*\gap+\hot,\vot) coordinate (a);
\draw(8*\wit+11*\gap+\hot+\wit,\vot) coordinate (b);
\draw(8*\wit+11*\gap+\hot,\vot+\het) coordinate (c);
\draw(8*\wit+11*\gap+\hot+\wit,\vot+\het) coordinate (d);
\draw(8*\wit+11*\gap+\hot,\vot+2*\het) coordinate (e);
\draw(8*\wit+11*\gap+\hot+\wit,\vot+2*\het) coordinate (f);
\draw(8*\wit+11*\gap+\hot,\vot+3*\het) coordinate (g);
\draw(8*\wit+11*\gap+\hot+\wit,\vot+3*\het) coordinate (h);
\draw(8*\wit+11*\gap+\hot,\vot+4*\het) coordinate (i);
\draw(8*\wit+11*\gap+\hot+\wit,\vot+4*\het) coordinate (j);
\draw(8*\wit+11*\gap+\hot,\vot+5*\het) coordinate (k);
\draw(8*\wit+11*\gap+\hot+\wit,\vot+5*\het) coordinate (l);
\draw(8*\wit+11*\gap+\hot,\vot+5.35*\het) coordinate (m);
\draw(8*\wit+11*\gap+\hot+0.5*\wit,\vot+5.35*\het) coordinate (s);
\draw(8*\wit+11*\gap+\hot+\wit,\vot+5.35*\het) coordinate (n);
\draw(8*\wit+11*\gap+\hot+0.5*\wit,\vot+3.2*\het) coordinate (o);
\draw(8*\wit+11*\gap+\hot+0.5*\wit,\vot+3.6*\het) coordinate (q);
\draw(8*\wit+11*\gap+\hot+0.2*\wit,\vot+3.6*\het) coordinate (p);
\draw(8*\wit+11*\gap+\hot+0.8*\wit,\vot+3.6*\het) coordinate (r);
\draw(8*\wit+11*\gap+\hot+0.5*\wit,\vot+4.65*\het) coordinate (ell);
\draw(8*\wit+11*\gap+\hot+0.5*\wit,\vot+0.5*\het) coordinate (cp0);
\draw(8*\wit+11*\gap+\hot+0.5*\wit,\vot+1.5*\het) coordinate (cp1);
\draw(8*\wit+11*\gap+\hot+0.5*\wit,\vot+2.5*\het) coordinate (cp2);
\draw[thick,red]  (o)--(p);\draw[thick,red]  (o)--(q); \draw[thick,red]  (o)--(r);
\draw[thick,red,dashed] (ell) ellipse (7pt and 11pt);
\draw[thick] (a)--(b)--(d)--(c)--(a);
\draw[thick]  (c)--(e)--(f)--(d);
\draw[thick]  (e)--(g)--(h)--(f);
\draw[thick]  (g)--(i)--(j)--(h);
\draw[thick]  (i)--(k)--(l)--(j);
\draw[thick]  (k)--(m)--(n)--(l);
\lcirc {o};\lcirc{q};\lcirc{p};\lcirc{r};
\nodeutxt{s}{$\fbx8{322}$};
\draw[thick,red,dotted] (cp0) ellipse (7pt and 9pt); \nodettxt{cp0}3;
\draw[thick,red,dotted] (cp1) ellipse (8pt and 5pt); \nodettxt{cp1}2;
\draw[thick,red,dotted] (cp2) ellipse (8pt and 5pt); \nodettxt{cp2}2;
%
%
\draw(9*\wit+12*\gap+\hot,\vot) coordinate (a);
\draw(9*\wit+12*\gap+\hot+\wit,\vot) coordinate (b);
\draw(9*\wit+12*\gap+\hot,\vot+\het) coordinate (c);
\draw(9*\wit+12*\gap+\hot+\wit,\vot+\het) coordinate (d);
\draw(9*\wit+12*\gap+\hot,\vot+2*\het) coordinate (e);
\draw(9*\wit+12*\gap+\hot+\wit,\vot+2*\het) coordinate (f);
\draw(9*\wit+12*\gap+\hot,\vot+3*\het) coordinate (g);
\draw(9*\wit+12*\gap+\hot+\wit,\vot+3*\het) coordinate (h);
\draw(9*\wit+12*\gap+\hot,\vot+4*\het) coordinate (i);
\draw(9*\wit+12*\gap+\hot+\wit,\vot+4*\het) coordinate (j);
\draw(9*\wit+12*\gap+\hot,\vot+5*\het) coordinate (k);
\draw(9*\wit+12*\gap+\hot+\wit,\vot+5*\het) coordinate (l);
\draw(9*\wit+12*\gap+\hot,\vot+5.35*\het) coordinate (m);
\draw(9*\wit+12*\gap+\hot+0.5*\wit,\vot+5.35*\het) coordinate (s);
\draw(9*\wit+12*\gap+\hot+\wit,\vot+5.35*\het) coordinate (n);
\draw(9*\wit+12*\gap+\hot+0.5*\wit,\vot+3.2*\het) coordinate (o);
\draw(9*\wit+12*\gap+\hot+0.5*\wit,\vot+3.6*\het) coordinate (q);
\draw(9*\wit+12*\gap+\hot+0.2*\wit,\vot+3.6*\het) coordinate (p);
\draw(9*\wit+12*\gap+\hot+0.8*\wit,\vot+3.6*\het) coordinate (r);
\draw(9*\wit+12*\gap+\hot+0.5*\wit,\vot+4.65*\het) coordinate (ell);
\draw(9*\wit+12*\gap+\hot+0.5*\wit,\vot+0.5*\het) coordinate (cp0);
\draw(9*\wit+12*\gap+\hot+0.5*\wit,\vot+1.5*\het) coordinate (cp1);
\draw(9*\wit+12*\gap+\hot+0.5*\wit,\vot+2.5*\het) coordinate (cp2);
\draw[thick,red]  (o)--(p);\draw[thick,red]  (o)--(q); \draw[thick,red]  (o)--(r);
\draw[thick,red,dashed] (ell) ellipse (7pt and 11pt);
\draw[thick] (a)--(b)--(d)--(c)--(a);
\draw[thick]  (c)--(e)--(f)--(d);
\draw[thick]  (e)--(g)--(h)--(f);
\draw[thick]  (g)--(i)--(j)--(h);
\draw[thick]  (i)--(k)--(l)--(j);
\draw[thick]  (k)--(m)--(n)--(l);
\lcirc {o};\lcirc{q};\lcirc{p};\lcirc{r};
\nodeutxt{s}{$\fbx9{232}$};
\draw[thick,red,dotted] (cp0) ellipse (8pt and 5pt); \nodettxt{cp0}2;
\draw[thick,red,dotted] (cp1) ellipse (7pt and 9pt); \nodettxt{cp1}3;
\draw[thick,red,dotted] (cp2) ellipse (8pt and 5pt); \nodettxt{cp2}2;
%
%
\draw(10*\wit+13*\gap+\hot,\vot) coordinate (a);
\draw(10*\wit+13*\gap+\hot+\wit,\vot) coordinate (b);
\draw(10*\wit+13*\gap+\hot,\vot+\het) coordinate (c);
\draw(10*\wit+13*\gap+\hot+\wit,\vot+\het) coordinate (d);
\draw(10*\wit+13*\gap+\hot,\vot+2*\het) coordinate (e);
\draw(10*\wit+13*\gap+\hot+\wit,\vot+2*\het) coordinate (f);
\draw(10*\wit+13*\gap+\hot,\vot+3*\het) coordinate (g);
\draw(10*\wit+13*\gap+\hot+\wit,\vot+3*\het) coordinate (h);
\draw(10*\wit+13*\gap+\hot,\vot+4*\het) coordinate (i);
\draw(10*\wit+13*\gap+\hot+\wit,\vot+4*\het) coordinate (j);
\draw(10*\wit+13*\gap+\hot,\vot+5*\het) coordinate (k);
\draw(10*\wit+13*\gap+\hot+\wit,\vot+5*\het) coordinate (l);
\draw(10*\wit+13*\gap+\hot,\vot+5.35*\het) coordinate (m);
\draw(10*\wit+13*\gap+\hot+0.5*\wit,\vot+5.35*\het) coordinate (s);
\draw(10*\wit+13*\gap+\hot+\wit,\vot+5.35*\het) coordinate (n);
\draw(10*\wit+13*\gap+\hot+0.5*\wit,\vot+3.2*\het) coordinate (o);
\draw(10*\wit+13*\gap+\hot+0.5*\wit,\vot+3.6*\het) coordinate (q);
\draw(10*\wit+13*\gap+\hot+0.2*\wit,\vot+3.6*\het) coordinate (p);
\draw(10*\wit+13*\gap+\hot+0.8*\wit,\vot+3.6*\het) coordinate (r);
\draw(10*\wit+13*\gap+\hot+0.5*\wit,\vot+4.65*\het) coordinate (ell);
\draw(10*\wit+13*\gap+\hot+0.5*\wit,\vot+0.5*\het) coordinate (cp0);
\draw(10*\wit+13*\gap+\hot+0.5*\wit,\vot+1.5*\het) coordinate (cp1);
\draw(10*\wit+13*\gap+\hot+0.5*\wit,\vot+2.5*\het) coordinate (cp2);
\draw[thick,red]  (o)--(p);\draw[thick,red]  (o)--(q); \draw[thick,red]  (o)--(r);
\draw[thick,red,dashed] (ell) ellipse (7pt and 11pt);
\draw[thick] (a)--(b)--(d)--(c)--(a);
\draw[thick]  (c)--(e)--(f)--(d);
\draw[thick]  (e)--(g)--(h)--(f);
\draw[thick]  (g)--(i)--(j)--(h);
\draw[thick]  (i)--(k)--(l)--(j);
\draw[thick]  (k)--(m)--(n)--(l);
\lcirc {o};\lcirc{q};\lcirc{p};\lcirc{r};
\nodeutxt{s}{$\fbx{10}{223}$};
\draw[thick,red,dotted] (cp0) ellipse (8pt and 5pt); \nodettxt{cp0}2;
\draw[thick,red,dotted] (cp1) ellipse (8pt and 5pt); \nodettxt{cp1}2;
\draw[thick,red,dotted] (cp2) ellipse (7pt and 9pt); \nodettxt{cp2}3;
%
%
\end{tikzpicture}
\caption{Copies of $A:=[16]$. Here $|B_0|=\dots=|B_4|=3$. 
In the copy $\fbx0{\emptyset}$, the oval stands for a 7-element subset. In each other copies,  the total number of elements in the ovals is also $7$.}
\label{figtwo}
\end{figure}

Next, we turn our attention to $\loS(W,n)$. Let $m:=\lint{n/3}$. For $n\notin\set{3,5,7}$, let $\qum:=\lint{(n-1)/2}$. For 
 $n \in\set{3,5,7}$,  $\qum$ stands for ${(n-3)/2}$. With $A:=[n]$, let us fix pairwise disjoint 3-element subsets $B_0$, $B_1$, \dots, $B_{m-1}$ of $A$, and denote the ``remainder set'' $A\setminus(B_0\cup \dots\cup B_{m-1})$ by $R$. 
These subsets are visualized in Figure \ref{figtwo}, where $n=16$,  $m=5$, $\qum=7$, and $A$ with subscripts and superscripts is drawn eleven times (in four groups separated by spaces). We can assume that $n\geq 3$. For $i\in\set{0,\dots,m-1}$, the elements of $B_i$ are denoted as follow: $B_i=\set{c_i,d_i,e_i}$.
For $i\in \Nnul$, call a 
vector $\vv=(v_0,\dots,v_{i-1})\in\set{2,3}^i$
\emph{eligible} if $i\leq m-1$ and 
$v_0+v_1+\dots+v_{i-1}\leq \qum$.   Note that for $i=0$, the empty vector is denoted by $\emptyset$ and it is eligible. As Figure \ref{figtwo} shows, there are exactly eleven eligible vectors for $n=16$; they  are the lower subscripts of the copies of $A$; because of space consideration, we write $232$ instead of $(2,3,2)$, etc., in the figure. (The upper subscripts of $A$ help to count the copies but play no other role.)

For each eligible $\vv$, we define a family of copies of $W$ in $\nPow =\Pow A$ as follows. Let $i$ denote the dimension of $\vv$, that is, $\vv=(v_0,\dots,v_{i-1})$. 
 For $j=0,\dots, i-1$, pick a $v_j$-element subset $X_j$ of $B_j$. In the figure, $X_j$ is denoted by a dotted oval with $v_j$ sitting in its middle.  
 Furthermore, pick a subset $X_i$ of $A\setminus (B_{0}\cup B_{1}\cup\dots\cup B_{i})$ such that $X_i=  \qum - v_0-\dots-v_{i-1}$. In the figure, $X_i$ is the dashed oval (without any number in its middle).  Let us emphasize that $X_j\subseteq B_j$ holds only for $j<i$ but it never holds for $j=i$.
Denote $(X_0, X_1, \dots, X_{i})$ by $\vX$, call it an \emph{eligible set vector}, and let $Z_{\vX}:=X_0\cup \dots \cup X_i$.  Clearly,
\begin{equation}
\parbox{7.5cm}{regardless the choice of $\vv$ and $\vX$, we have that $|Z_{\vX}|$ is always the same, namely, $ |Z_{\vX}|=\qum$.}
\label{eq:mklfKcsBrzntRnld}
\end{equation}
For convenience, let $\setc_i:=\set{c_i}$, $\setd_i:=\set{d_i}$, $\sete_i:=\set{e_i}$, and $\setz_i:=\emptyset$. 
Observe that $\set{\setc_i, \setd_i, \sete_i,\setz_i}$ is a copy of $W$ in $\Pow{B_i}$; in each copy of $A$ in the figure, this copy of $W$ is indicated by its diagram for exactly one $i$. 
It follows that 
\begin{equation}
\begin{aligned}
&W_{\vX}:=\set{ z_{\vX},\,\,   c_{\vX},\,\,   d_{\vX},\,\,   e_{\vX} }\text{, where }  \cr
&z_{\vX}:=\setz_i\cup Z_{\vX}=Z_{\vX},\,\,   c_{\vX}:=\setc_i\cup Z_{\vX},\,\,   d_{\vX}:=\setd_i\cup Z_{\vX},\,\,   e_{\vX}:=\sete_i\cup Z_{\vX}, 
\end{aligned}
\end{equation}
is also a copy of $W$ but now in $\Pow A=\nPow$.

To prove that $\loS(W,n)\leq \Sp(W,n)$, we need to show that $\loS(W,n)$ is the number of eligible set vectors $\vX$ and for distinct eligible set vectors $\vX\neq \vpX$, the corresponding copies $W_{\vX}$ and $W_{\vpX}$ of $W$ are unrelated. 

First, we deal with the number of eligible set vectors $\vX=(X_0,\dots, X_i)$. As each of the $B_j$'s are 3-element and there are $\lint{n/3}$ many of them, the largest value of $i$ is at most $\lint{n/3}-1$, the upper limit of the outer summation index in \eqref{eq:rtrGtvZTlP}. The eligible vector $\vv$ that gives rise to $\vX$ is uniquely determined by $\vX$ since $\vv=(|X_0|,\dots,|X_{i-1}|)$. Let $j:=|\set{t\in\set{0,\dots, i-1}: v_t=2}|$. This $j$, which corresponds to the inner summation index in \eqref{eq:rtrGtvZTlP},  is the number of $2$'s in dotted ovals in the figure. 
There are $\tbinom{i}{j}$ possibilities to choose the $j$-element set $\set{t\in\set{0,\dots, i-1}: v_t=2}$; this is where the first binomial coefficient enters into \eqref{eq:rtrGtvZTlP}. For each $t\in\set{0,\dots, i-1}$ such that $v_t=2$,  we can choose the 2-element subset $X_t$ of $B_t$ in 3 ways. As there are $j$ such $t$'s, this brings the power $3^j$ into \eqref{eq:rtrGtvZTlP}.  
Since $X_i$ is a subset of the $n-3i-3$-element set $A\setminus (B_0\cup\dots\cup B_i)$ and 
\begin{align*}
|X_i|=\qum-v_0-\dots-v_{i-1}
=\qum-2j-3(i-j)=\qum+j-3i\,,
\end{align*}
the second binomial coefficient in \eqref{eq:rtrGtvZTlP} gives how many ways we can choose $X_i$. 
Therefore, \eqref{eq:rtrGtvZTlP} precisely gives the number of eligible set vectors $\vX$.

Next, assume that $\vX=(X_0,\dots, X_i)$ and $\vpX=(\bullX_0,\dots,\bullX_{\bulli})$ are distinct eligible set vectors with corresponding (not necessarily different) eligible vectors $\vv=(v_0,\dots,v_{i-1})$ and $\vpv=(\bullv_0,\dots,\bullv_{\bulli-1})$. Assume also that $\sima$ and $\bulla$ are in $W$ such that 
$(\vX,\sima)\neq (\vpX,\bulla)$. 
We need to show that  $\sima_{\vX}=\seta_i\cup  Z_X$ and $\bulla_{\vpX}= \bullseta_i\cup   Z_{\vpX}$
are incomparable. 
There are two cases to consider; both can  easily be followed by keeping an eye on Figure \ref{figtwo} in addition to the formal argument.

First, assume that $i\neq \bulli$, say, $i<\bulli$.  Observe that 
$|\bulla_{\vpX} \cap B_i|=|\bullX_i|=\bullv_i\geq 2$ but $|a_{\vX} \cap B_i|=|\seta_i|\leq 1$. So $|\bulla_{\vpX} \cap B_i| > |a_{\vX} \cap B_i|$.
(Pictorially, a dotted oval, labeled by 2 or 3, has more element than $|\seta_i|$ symbolized by one of the vertices of the diagram of W drawn in $B_i$.) Hence,  $\bulla_{\vpX}\nsubseteq a_{\vX}$. For the sake of contradiction, suppose that  $a_{\vX} \subseteq \bulla_{\vpX}$.
Then for every $j\in\set{0,\dots,i-1}$,
$v_j=|B_j\cap a_{\vX}| \leq |B_j\cap \bulla_{\vpX} |=\bullv_j$. Hence, we can compute as follows; the computation is motivated by comparing, say, $\fbx1{2}$ and $\fbx9{232}$ in Figure \ref{figtwo}:
\begin{align*}
|a_{\vX} &\cap (B_{i+1}\cup\dots \cup B_{m-1}\cup R)|=|X_i| \cr
&= \qum - v_0-\dots-v_{i-1} \geq  \qum - \bullv_0-\dots-\bullv_{i-1}\cr
&=  (\qum - \bullv_0-\dots-\bullv_{\bulli-1}) +
(\bullv_{i+1}+\dots+ \bullv_{\bulli-1}) + |\bullseta_{\bulli}| + (\bullv_i -  |\bullseta_{\bulli}| ) \cr
&=|\bullX_{\bulli}| + (|\bullX_{i+1}|+\dots+|\bullX_{\bulli-1}|  + |\bullseta_{\bulli}|) + (\bullv_i -  |\bullseta_{\bulli}| )\cr
&=|\bulla_{\vpX} \cap (B_{i+1}\cup\dots \cup B_{m-1}\cup R)| +  (\bullv_i -  |\bullseta_{\bulli}| ) \cr
&> |\bulla_{\vpX} \cap (B_{i+1}\cup\dots \cup B_{m-1}\cup R)|.
\end{align*}
The strict inequality just obtained contradicts that $a_{\vX} \subseteq \bulla_{\vpX}$, and we conclude in the first case that $a_{\vX}\parallel \bulla_{\vpX}$, as required.

Second, assume that $i =\bulli$. If 
$X_j\parallel \bullX_j$ for some $j\in \set{0,\dots,i}$ or  there are $s,t\in\set{0,\dots,i}$ such that\footnote{According to the convention of lattice theory, ``$\subset$'' is the conjunction of ``$\subseteq$'' and ``$\neq$''.} 
$X_s\subset \bullX_s$ but $X_t\supset \bullX_t$, then 
the validity of  $a_{\vX}\parallel \bulla_{\vpX}$ is clear. 
Thus, we can assume that $X_j\subseteq \bullX_j$ for all $j\in\set{0,\dots, i}$. Then 
 \[
 \qum=|X_0|+\dots+|X_i|\leq |\bullX_0|+\dots+|\bullX_i| = \qum
 \]
together with $|X_j| \leq |\bullX_j|$, for all $j\in\set{0,\dots, i}$, 
imply that $X_j=\bullX_j$ for all $j\in\set{0,\dots,i}$. Combining this equality with 
 $X_j\subseteq \bullX_j$ for all $j\in\set{0,\dots, i}$, we obtain that  $\vX=\vpX$,  contradicting our assumption.  We have shown that $\loS(W,n)\leq \Sp(W,n)$, as required.

It is well known that no matter how we fix two integers $s$ and $t$, 
\begin{equation}
\binom{n-s}{\lint{n/2}-t}\text{ is asymptotically } 2^{-s}\binom n{\lint{n/2}}=2^{-s}\fsb(n)\text{ if }n\to\infty;
\label{eq:lkVtnrmgrStJrlgT}
\end{equation} 
this folkloric (and trivial) fact was used in Dove and Griggs \cite{dovegriggs} and Katona and Nagy \cite{KatonaNagy}, too. This fact and \eqref{eq:cjWhlmgTlP} yield that $\upS(W,-)$ is asymptotically $\frac14 \fsb(-)$. Hence, to obtain the required asymptotic equations, it suffices to show that 
$\loS(W,-)$ is asymptotically $\frac14 \fsb(-)$, too. 
Let $\eta$ and $\mu$ be small positive real numbers.
As $\sum_{i=0}^\infty 2^{-i}=2$, we can fix 
a $q\in\Nplu$ such that  $\sum_{i=0}^{q} 2^{-i}\geq 2- \eta$. Using \eqref{eq:lkVtnrmgrStJrlgT} and assuming that $i\leq q$, we obtain that the second binomial coefficient in \eqref{eq:rtrGtvZTlP} asymptotically  $2^{-3i}\fsb(n-3)$ or, rather, it is $\frac18\cdot 2^{-3i}\fsb(n)$. So it is at least $\frac18\cdot 8^{-i}(1-\mu)\fsb(n)$ for all but finitely many $n$. Hence, assuming $n$ is large enough and, in particular, $\lint{n/3} >q$, 
\allowdisplaybreaks{
\begin{align}
\loS(W,n)&\geq \sum_{i=0}^q\sum_{j=0}^i 3^j \binom ij \cdot 8^{-i}\cdot \frac 18 (1-\mu)\fsb(n)\cr
&= \frac18 (1-\mu)\fsb(n) \sum_{i=0}^q 8^{-i} \sum_{j=0}^i \binom ij  3^j \cdot 1^{i-j}\cr
&=\frac18 (1-\mu)\fsb(n) \sum_{i=0}^q 8^{-i} (3+1)^i\cr
&\geq \frac18 (1-\mu)\fsb(n)(2-\eta)=\frac{(2-\eta)(1-\mu)}8 \fsb(n).
\label{eq:nhMrlkRsvRm}
\end{align}}
As the  last fraction in \eqref{eq:nhMrlkRsvRm} can be arbitrarily close to $1/4$, it follows that  $\loS(W,n)$  is asymptotically at least $\frac 14\fsb(n)$. It is asymptotically at most  $\frac 14\fsb(n)$ since so is $\upS(W,n)$ and we  know that $ \loS(W,n)\leq \Sp(W,n)\leq \upS(W,n)$.  This completes the argument proving the ``asymptotically equal'' part of Proposition \ref{prop:Wposet}.

Next, we turn our attention to the left adjoints of our estimates. First of all, we claim that 
\begin{equation}
\text{for every }n\in\Nplu, \quad \upS(W,n)\leq \loS(W,n+1).
\label{eq:lNsGlRZndhBl}
\end{equation}
Let $\uppS(W,-)$ be the same as $\upS(W,-)$ except that we drop the outer ``lower integer part'' function from its definition. It suffices to prove \eqref{eq:lNsGlRZndhBl} with $\uppS(W,n+1)$ instead of $\upS(W,n+1)$.  We can assume that $n\geq 10$ as otherwise \eqref{eq:lNsGlRZndhBl} is clear by Table \ref{tableWsmall}\footnote{\label{foot:ryzen}The table was obtained by the computer algebraic program Maple V Release 5, which ran  on a desktop computer with AMD Ryzen 7 2700X Eight-Core Processor 3.70 GHz for  $\frac 15$ seconds.}. Let $T(n)$ denote the sum of the two summands in the upper line of \eqref{eq:rtrGtvZTlP} that correspond to $(i,j)=(0,0)$ and $(i,j)=(1,1)$.  After a straightforward but  tedious calculation, if $n=2m$, then 
\begin{equation} \frac{\uppS(W,n)}{\loS(W,n+1)} \leq  \frac{\uppS(W,n)}{T(n+1)}=   \frac{2m(2m-2)(2m-3) } {(m-1)^2(11m-12)}.
\label{eq:kRsmkmVlglDhGk}
\end{equation}
Subtracting the numerator from the denominator, we obtain
$3m^3-14m^2+23m-12$, which is clearly nonnegative for $5\leq m\in \Nplu$ (in fact, for all $m\in\Nplu$), whence  
the fraction is at most 1 for $n=2m\geq 10$.
For and odd $n=2n+1 \geq 10$, \eqref{eq:kRsmkmVlglDhGk} turns into 
\begin{equation*} \frac{\uppS(W,n)}{\loS(W,n+1)} \leq  \frac{\uppS(W,n)}{T(n+1)}=   \frac{4(2m+1)(2m-1)(2m-3)} {(4m+1)(11m^2-19m+6)}\,,
\end{equation*}
and now the subtraction gives the polynomial $12m^3-17m^2+13m-6$, which is clearly nonnegative for $2\leq m\in\Nplu$ (in fact, for all $m\in\Nplu$). Thus, passing from $m$ to $n$, the required inequality  $\uppS(W,n)\leq \loS(W,n+1)$ holds for all $10\leq n\in\Nplu$.   We have shown the validity of \eqref{eq:lNsGlRZndhBl}.

\begin{table}
\allowdisplaybreaks{
\begin{gather*}
\begin{tabular}{l|r|r|r|r|r|r|r|r|r|r|r|r|r}
$n$ &  3& 4& 5& 6& 7& 8& 9 & 10 & 11 & 12 & 13 &14 \cr
\hline
$\loS(W,n)$ &\textcircled{\raisebox{-1pt}1}&\textcircled {\raisebox{-1pt}1}&\textcircled {\raisebox{-1pt}2}& \textcircled {\raisebox{-1pt}6}&9&17&36&66&120 & 234& 456& 876\cr
\hline
$\upS(W,n)$ &1&2 &3&6 &10&20&37&70&132&252&480& 924\cr
\hline
\end{tabular}
\\
 \begin{tabular}{l|r|r|r|r|r|r|r|r}
$n$ &  15& 16& 17& 18& 19& 20& 21\cr
\hline
$\loS(W,n)$ & 1\,680&3\,625& 6\,340 &12\,330&23\,960&46\,766&91\,224\cr
\hline
$\upS(W,n)$ & 1\,775&3\,432& 6\,630&12\,870&24\,967&48\,620&94\,631\cr
\hline
\end{tabular}
\\
\begin{tabular}{l|r|r|r|r|r|r}
$n$ &  22& 23& 24& 25& 26\cr
\hline
$\loS(W,n)$ &178\,388&348\,656&683\,130&1\,337\,896\,&2\,625\,364\cr
\hline
$\upS(W,n)$ &184\,756&360\,554&705\,432&1\,379\,671&2\,704\,156\cr
\hline
\end{tabular} 
\\
\begin{tabular}{l|r|r|r|r|r}
$n$ &  27& 28& 29& 30\cr
\hline
$\loS(W,n)$ &5\,149\,872&10\,119\,348&19\,877\,904&39\,104\,856\cr
\hline
$\upS(W,n)$ &5\,298\,418&10\,400\,600&20\,410\,200&40\,116\,600\cr
\end{tabular}
\end{gather*}
}%
\caption{Some values of $\loS(W,n)$ and $\upS(W,n)$; the \emph{known} values of $\Sp(W,n)$ are encircled.}
\label{tableWsmall}
\end{table}

Next, we deal with \eqref{eq:mLptpsGPs}. By Table \ref{tableWsmall}, the first few values of $\upS^\ast(W,k)$ and those of  $\loS^\ast(W,k)$ are as follows:
\begin{equation}
\begin{tabular}{l|r|r|r|r|r|r|r|r|r|r|r|r|r|r|r|r}
$k$ & 1& 2& 3& 4&5&6&7&8&9&10&11&12&13&14&15\cr
\hline
$\upS^\ast(W,k)$ &3&4&5&6&6&6&7&7&7 &7&8&8&8&8&8 \cr
\hline
$\loS^\ast(W,j)$ & 3&5&6&6& 6& 6& 7& 7& 7&8&8&8&8&8&8\cr
\end{tabular}
\label{eq:mFljZrstlWlFj}
\end{equation}
This implies  \eqref{eq:mLptpsGPs} for $k\leq 15$ (in fact, for $k\leq 29$),  so we can assume that $k>15$. Using \eqref{eq:mFljZrstlWlFj} and the obvious fact that $\loS(W,-)$ is a strictly increasing function on $\Nplu\setminus[7]$, there is a unique $7\leq n\in\Nplu$ such that 
\begin{equation}
 \loS(W,n)<k\leq \loS(W,n+1).
\end{equation}
Using  \eqref{eq:lNsGlRZndhBl} and the inequality $\loS(W,n)\leq \upS(W,n)$, we obtain that 
\begin{align}
\text{either }\loS(W,n)<  k \leq \upS(W,n)
\label{eq:rvzhbkRkgpRc}\\
\text{or }  \upS(W,n)<k\leq \loS(W,n+1).
\label{eq:nHnGrcgZlN}
\end{align}
If  \eqref{eq:rvzhbkRkgpRc},  then 
$\upS^\ast(V,k)=n$ and $\loS^\ast(W,k)=n+1$.
If \eqref{eq:nHnGrcgZlN}, then $\upS^\ast(W,k)=n+1=\loS^\ast(V,k)$.
In both cases, $0\leq \loS^\ast(W,k)-\upS^\ast(W,k)\leq 1$, as required.

Next, for $t\in \Nplu$, let $E_t:=\set{k\in [t]: \upS^\ast(W,k)< \loS^\ast(W,k)}$.
To settle the last sentence of Proposition \ref{prop:Wposet} about density,  it suffices to show that $\lim_{t\to\infty} (|E_t|/t) = 0$.  
Let $\epsilon<1/12$ be a  positive real number; we are going to show that $|E_t|/t<\epsilon$ for all but finitely many $t$'s. Asymptotic equalities will often be denoted by ``$\sim$''. As we have already proved that $\loS(W,-)\sim\frac 14\fsb(-)$, \eqref{eq:lkVtnrmgrStJrlgT} yields that $\upS(W,n-1)/\upS(W,n)\to 1/4$ as $n\to \infty$. 
This fact, $1/6<1/4<3^{-1}$,  and   $\loS(W,n)\sim\upS(W,n)$ allow us to fix an $n_0=n_0(\epsilon)\in \Nplu$ such that for all $n\geq n_0$, 
\begin{align}
&\upS(W,n)/6< \upS(W,n-1)< \upS(W,n)\cdot3^{-1}
\,\,\text{ and }
\label{eq:mRgKrblFcsK}
\\
&\upS(W,n) - \loS(W,n)<
\upS(W,n)\cdot  \epsilon / 12\,\,.
\label{eq:sZgfJbkrtSj}
\end{align} 
Later, it will be important that  $n_0$ does not depend on $t$. Hence, from now on, we can assume that $\upS(W,n_0)<t$.  Since $\lim_{i\to\infty}\upS(n_0+i)=\infty$ in a strictly increasing way, there exists a unique
$r=r(t)\in \Nplu$ such that $\upS(n_0+r-1)<t\leq \upS(n_0+r)$. Since $\epsilon$ is small, \eqref{eq:mRgKrblFcsK} and \eqref{eq:sZgfJbkrtSj} yield that  for all $i\in[r]$,   
\begin{align}
&\underbrace{\upS(W,n_0+i-1)< \loS(W,n_0+i)}_{\text{long good interval}}\leq   \upS(W,n_0+i)  
\label{eq:bkpCskNcrZ}\\
&\text{and }\quad
\upS(W,n_0+r) / 6 < t\,.
\label{eq:rCtnBtkvTgngRck}
\end{align}

Observe that \eqref{eq:bkpCskNcrZ} and $\loS(W,n_0+i-1)\leq \upS(W,n_0+i-1)$ imply that  for every $k$ in the left open and right closed  interval $\bigl(\upS(W,n_0+i-1), \loS(W,n_0+i)\bigr]$, which is under-braced in \eqref{eq:bkpCskNcrZ},  $\loS^\ast(W,k)=\upS^\ast(W,k)=n_0+i$. So this interval is disjoint from $|E_t|$ for any $t\in\Nplu$. 
Thus, letting $c:=\upS(W,n_0)$ and $d:=\upS(W,n_0+r)$,  $E_t$ is a subset of $[1,c]\cup\bigcup_{i\in[r]} \bigl(\loS(W,n_0+i), \upS(W,n_0+i)\bigr]$. Hence, 
\begin{align*}|&E_t|\leq c+\sum_{i\in[r]}\bigl(\upS(W,n_0+i) - \loS(W,n_0+i)\bigr)
\cr
&
\rleqref{eq:sZgfJbkrtSj}
 c+\frac\epsilon {12} \cdot \sum_{i\in[r]} \upS(W,n_0+i) 
=
c+\frac\epsilon {12} \cdot \sum_{i\in\set{0,\dots,r-1}} \upS(W,n_0+r-i)
\cr
&\rleqref{eq:mRgKrblFcsK} c+\frac\epsilon {12}   \sum_{i\in\set{0,\dots,r-1}}3^{-i}d 
\leq c+\frac{\epsilon d} {12}  \sum_{i\in \Nnul}3^{-i} = c+ \frac{\epsilon d} {12}\cdot\frac 4 3 = c+\frac{\epsilon d}9.
\end{align*} 
This inequality and \eqref{eq:rCtnBtkvTgngRck} yield that $|E_t|/t\leq |E_t|/\bigl(\upS(W,n_0+r)/6\bigr) \leq (c+\epsilon d/9)/(d/6)
=6c/d + 2\epsilon/3$. As $t\to \infty$, $r=r(t)$ and  $d=\upS(W,n_0+r)$ also tend to $\infty$. So for all sufficiently large $t$, we have that $6c/d<\epsilon/3$, whereby  $|E_t|/t\leq \epsilon/3 + 2\epsilon/3=\epsilon$. 
Thus, $0\leq |E_t|/t <\epsilon$ for all  but finitely many $t$, and this is true for every positive $\epsilon \leq 1/12$.  
That is,  $\lim_{t\in\Nplu} |E_t|/t =0$. Hence, the natural density of $E$ is 0 and that of $\Nplu\setminus E$, which occurs in Proposition \ref{prop:Wposet}, is 1, as required. 
The proof or Proposition \ref{prop:Wposet} is complete.
\end{proof}

\begin{remark}[Differences from \cite{dovegriggs} and  \cite{KatonaNagy}]\label{ref:diffDGKN} The differences we are going to summarize here are partly due to the fact that, naturally, more can be proved for a small particular poset  than for all finite posets. When proving that $\Sp(W,n)\leq \upS(W,n)$,  the only novelty is the argument between \eqref{eq:slBsjmGsTlt} and \eqref{eq:mTlThlrKtmklmSrrnd}. More novelty occurs in our proof of $\loS(W,p)\leq \Sp(W,n)$.  
As opposed to Dove and Griggs \cite{dovegriggs},where several ``layers'' are populated,  we use no iteration and we have \eqref{eq:mklfKcsBrzntRnld}. Compared to Katona and Nagy \cite{KatonaNagy}, our construction performs better for small values of $n$; the following table shows what lower estimates could be extracted from \cite{KatonaNagy}.
\[
\begin{matrix}
n&10&50&100
\cr\hline
\text{by \cite{KatonaNagy}:}&21& 14\,833\,897\,694\,226 & 
12\,229\,253\,884\,310\,811\,313\,310\,605\,728
\cr
\hline
\loS(W,n):&66& 31\,761\,385\,392\,516 &
25\,286\,044\,048\,404\,745\,303\,553\,386\,716
\end{matrix}
\]
 (We have no similar numerical comparison in case of \cite{dovegriggs}.) Except for \eqref{eq:knknRswt}, which is  quoted from \cite{KatonaNagy} and does not apply for  $W$,  \cite{dovegriggs} and 
 \cite{KatonaNagy} give only asymptotic results but no concrete values of $\Sp(U,n)$ for a poset $U$. 
\end{remark}

\begin{remark}\label{rem:lRglnfn}
Even for a small $n$, the trivial algorithm for determining $\Sp(W,n)$ is far from being feasible. For example, for $n=10$,  the ``cover-preserving'' copies of $W$ in $\Pow{[10]}$ form a  $\sum_{i=0}^7\ibinom{10}{i}\cdot\ibinom{10-i}{3}=15\,360$-element set $\mathcal H$. All the 
 $(\Sp(W,10)+1)$-element subsets of $\mathcal H$  should be excluded, but    no computer can exclude $\ibinom{15\,360}{\Sp(10)+1}\geq \ibinom{15\,360}{67}\geq 10^{185}$ subsets; the first inequality here comes from  Table \ref{tableWsmall}.
\end{remark}

We have investigated another small poset, too; it is the 3-element non-chain poset $V$; see Figure \ref{figone}. Define
\begin{align}
\loS(V,n):=\sum_{i=0}^{\lint{\uint{(n-2)/2}/2}}
\binom{n-2-2i}{\uint{(n-2)/2}-2i}\quad \text{ and}
\label{eq:slsblTlP}
\\
\upS(V,n):=\Bigl(1+ \frac{2n-3\lint{n/2}-1}
{2n-\lint{n/2}-1}\Bigr)\cdot \binom{n-2}{\lint{(n-2)/2}}.
\label{eq:cssnblTlP}
\end{align}

\nothing{up to n=311, the shift parameter $b$ is optimally chosen!}

\nothing{For n=2023, the computation took 221 seconds. }

\begin{proposition}
[Mostly from Katona and Nagy \cite{KatonaNagy}]\label{prop:Vposet}
For $2\leq n\in\Nplu$,  Proposition \ref{prop:Wposet} remains valid if we substitute $V$ and $\frac13\fsb(-)$ for $W$ and $\frac14\fsb(-)$, respectively. 
\end{proposition}

A few values of $\loS(V,n)$ and $\upS(V,n)$ are listed below
\begin{align}
&\begin{tabular}{l|l|l|l|l|l|l|l|l|l|l|l|l|l}
$n$ & 2 & 3& 4& 5& 6& 7& 8& 9 & 10 & 11 & 12 &13 \cr
\hline
$\loS(V,n)$ & 1 & 1& 2& 4& 7& 13& 24& 46 & 86& 166& 314&610\cr
\hline
$\upS(V,n)$ & 1 & 1& 2& 4& 7& 14& 25& 48 & 90& 173& 326&632\cr
\end{tabular} \, ,
\label{eq:nbCnKkssljssmrZlj}\\
&\begin{tabular}{l|r|r|r|r|r}
$n$ &  14& 15&2022& 2023  \cr
\hline
$\loS(V,n)$ & 1\,163& 2\,269& $\approx 2.848\,220\cdot 10^{606}$ & $\approx 5.695\,500\cdot 10^{606}$ \cr
\hline
$\upS(V,n)$ &  1\,201& 2\,340& $\approx 2.848\,846\cdot 10^{606}$& $\approx 5.696\,752\cdot 10^{606} $ \cr
\hline
$^\text{up}S/{}_\text{lo}S\approx$ & $1.033$ & $1.031$&  $1.000\,219\,853$& $1.000\,219\,780$ \cr
\end{tabular}\,.
\label{eq:tvTkbglnrL}
\end{align}

We do not prove this proposition in the paper.  It would be straightforward to simplify the proof of Proposition \ref{prop:Wposet}  to obtain a proof of Proposition \ref{prop:Vposet}.  (The simplification means that $|B_i|=2$ and all the eligible vectors $\vv$ are of the form $(1,\dots,1)$ and so we do not need them.)
Note that \href{https://arxiv.org/abs/2308.15625v2}{arXiv:2308.15625v2}, the earlier version of this paper, contains a detailed proof  of  Proposition \ref{prop:Vposet}. However, our construction to prove that $\loS(V,n)\leq \Sp(V,n)$ is included already in Katona and Nagy \cite[last page]{KatonaNagy}, where $\loS(V,n) = \Sp(V,n)$ is conjectured. 
Note the little typo in  \cite[equation (27)]{KatonaNagy};  the upper limit of the summation should be $\lint{\frac{n+3}2}$ rather than $\lint{\frac{n+2}2}$. After that this typo is corrected, (27) in  \cite{KatonaNagy} is the same as \eqref{eq:slsblTlP}. 

Next, we give the following mini-table; the computation for it  took twelve minutes, see Footnote \ref{foot:ryzen},.
\begin{equation}
\begin{tabular}{l|r|r|r|r}
$n$ & 2022& 2023 &2024 \cr
\hline
$\loS(W,n)\approx$ &  $2.136\,194\cdot 10^{606}$ & $ 4.271\,332
\cdot 10^{606}$ & $  8.540\,554 \cdot 10^{606}$  \cr
\hline
$\upS(W,n)\approx$ & $2.136\,987
\cdot 10^{606}$& $ 
4.272\,916 \cdot 10^{606} $& $8.543\,720 \cdot 10^{606}$  \cr
\hline
$^\text{up}S/{}_\text{lo}S\approx$ &   $1.000\,371\,103$& $1.000\,370\,920$ &1.000\,370\,737   \cr
\end{tabular}
\label{eq:tvWnnKhnjNlr}
\end{equation}

It follows from Propositions \ref{prop:Wposet} and \ref{prop:Vposet}, Table \ref{tableWsmall}, \eqref{eq:tvgrHptnKhnjN}, \eqref{eq:mFljZrstlWlFj}, \eqref{eq:nbCnKkssljssmrZlj}, \eqref{eq:tvTkbglnrL}, and \eqref{eq:tvWnnKhnjNlr} that the minimum sizes of generating sets of the $k$-th direct powers of the lattices $\Dn V$ and $\Dn W$,  drawn in Figure \ref{figone},  and the 5-element chain $\chain 4$ are  given as follows.
\begin{equation}
\begin{tabular}{l|r|r|r|r|}
$k$ & 2022& 2023 &$3\cdot 10^{606}$ &$5\cdot 10^{606}$    \cr
\hline
$\gmin{\chain4^k}$ &  $18$ & $18$ &$2025$&2026\cr
\hline
$\gmin{D(V)^k}$ &  $15$ & $15$&2023 &2023 \cr\hline
$\gmin{D(W)^k}$ &  $16$ & $16$&2023&2024\cr
\end{tabular}\,\,.
\label{eq:thmwRtlzW}
\end{equation}

\section{Appendix: Maple worksheet}
In this section, we present the Maple worksheet that computed Table  \ref{tableWsmall}; see Footnote \ref{foot:ryzen}.   For the rest of the numerical data in the paper, either the two parameters in the
``\texttt{for n from 3 to 30 do}''  can be modified or a much simpler worksheet would do.
{\small
\begin{verbatim}
> restart;       time0:=time():
> #An upper bound for Sp(W,n):
> upSW:= proc(n) local s; s:=n/(3*n-2-2*floor(n/2));
>   floor(s*binomial(n-1, floor((n-1)/2))); 
> end: 
> # A lower bound for Sp(W,n):
> loSW:=proc(n) local s,i,j,ub,lb,h,summand,returnvalue; 
>   s:=0;  
>   if (n=3) or (n=5) or (n=7) then h:=floor((n-3)/2)
>                              else h:=floor((n-1)/2)
>   fi;
>   for i from 0 to ceil(n/3)-1 do ub:=n-3-3*i;
>    #ub: Upper number in the 2nd Binomial coefficient
>    if ub >= 0 then  
>     for j from 0 to i do lb:=h-2*j-3*(i-j);# j: number of 2's,
>      #lb: Lower number in the 2nd Binomial coefficient
>      if (lb>=0) and (lb<=ub) then 
>        summand:=binomial(i,j)*3^j*binomial(ub,lb); 
>        s:=s+summand;  
>      fi;#end of the "if (lb>=0) and (lb<=ub)" command  
>     od; #end of the  j loop
>    fi; #end of the "if ub >= 0" command 
>   od; #end of the i loop
>  returnvalue:=s; #the procedure returns with the last result 
> end:
> for n from 3 to 30 do lower:=loSW(n): 
>  upper:= upSW(n): 
>  if lower>0 then ratio:=evalf(upper/lower) else ratio:=undefined fi :
>  print(`n=`, n, ` lower=` ,lower, ` upper=`, 
>      upper, ` ratio=`, ratio);
>  if lower>10^6 then
>   print(`lg(lower)=`,evalf(log[10](lower)), 
>         `lg(upper)=`,evalf(log[10](upper))) fi;  
> od: 
> time2:=time():
> print(`The total computation needed `, time2-time0,` seconds.`);
\end{verbatim}
}

\end{document}